\documentclass[10pt,a4paper,oneside,reqno]{amsart}

\usepackage[totalwidth=14cm,totalheight=21.5cm]{geometry}
\usepackage[T1]{fontenc}
\usepackage[utf8]{inputenc}
\usepackage[english]{babel}

\usepackage{multicol}
\usepackage{enumitem}
\usepackage{mparhack}

\usepackage{amsmath}
\usepackage{amssymb}
\usepackage{amsthm}
\usepackage{bbm}
\usepackage{braket}

\usepackage{tikz-cd}
\usepackage{pgfplots}
\usepackage{color}

\usepackage[colorlinks=true,pagebackref=true]{hyperref}
\usepackage{cleveref}

\usepackage{lipsum,xpatch}

\theoremstyle{plain}
\newtheorem{teo}{Theorem}[section]
\newtheorem{pro}[teo]{Proposition}
\newtheorem{cor}[teo]{Corollary}
\newtheorem{lem}[teo]{Lemma}

\makeatletter
\newtheorem*{rep@theorem}{\rep@title}
\newcommand{\newreptheorem}[2]{%
	\newenvironment{rep#1}[1]{%
		\def\rep@title{#2 \ref{##1}}%
		\begin{rep@theorem}}%
		{\end{rep@theorem}}}
\makeatother

\newtheorem{thmx}{Theorem}
\newreptheorem{thmx}{Theorem}

\newtheorem{corx}[thmx]{Corollary}
\newreptheorem{corx}{Corollary}

\theoremstyle{remark}
\newtheorem{oss}[teo]{Remark}
\newtheorem{de}[teo]{Definition}
\newtheorem{es}[teo]{Example}

\makeatletter
\newcounter{step}
\xpretocmd{\proof}{\setcounter{step}{0}}{}{}
\newcommand{\step}[1]{%
	\par
	\addvspace{\medskipamount}%
	\stepcounter{step}%
	\noindent\emph{Step \thestep: #1.}\par\nobreak\smallskip
	\@afterheading
}
\makeatother

\newcommand{\R}{\mathbb{R}}
\newcommand{\C}{\mathbb{C}}

\newcommand{\Z}{\mathbb{Z}}
\newcommand{\N}{\mathbb{N}}

\newcommand{\hyp}{\mathbb{H}}
\newcommand{\hypd}{\mathbb{H}}
\newcommand{\hypu}{\widetilde{\mathbb{H}}}
\newcommand{\AdS}{\mathbb{A}\mathrm{d}\mathbb{S}}

\newcommand{\h}{\mathcal{H}}
\newcommand{\sph}{\mathbb{S}}

\newcommand{\dis}{\mathbb{D}}

\newcommand{\Span}{\operatorname{Span}}
\newcommand{\Isom}{\operatorname{Isom}}

\newcommand{\nablah}{\overline\nabla}

\newcommand{\sq}{\subseteq}
\newcommand{\pd}{\partial}
\newcommand{\pff}{\mathrm{I}}
\newcommand{\sff}{\mathrm{I\!I}}
\newcommand{\hess}{\operatorname{Hess}}
\newcommand{\gr}{\operatorname{graph}}
\newcommand{\osc}{\operatorname{osc}}
\newcommand{\pr}[1]{\langle #1\rangle}

\newcommand{\ch}{\mathcal{CH}}
\newcommand{\past}{\mathbf{P}}
\newcommand{\future}{\mathbf{F}}
\newcommand{\barrf}[1]{W^{\future}_{#1}}
\newcommand{\barrp}[1]{W^{\past}_{#1}}
\newcommand{\dist}{d_L}

\setbox257=\hbox{$\sup$}
\setbox258=\hbox{$\inf$}
\DeclareMathOperator*{\Inf}{\raisebox{0pt}[\ht258][\dp257]{$\inf$}}

\title[CMC hypersurfaces in $\hyp^{n,1}$]{Constant mean curvature hypersurfaces in Anti-de Sitter space}

\author[Enrico Trebeschi]{Enrico Trebeschi}
\address{Enrico Trebeschi: Università degli Studi di Pavia - Université Grenoble Alpes}
\email{enrico.trebeschi01@universitadipavia.it\\
enrico.trebeschi@univ-grenoble-alpes.fr}
\begin{document}

\begin{abstract}
	We study spacelike entire constant mean curvature hypersurfaces in Anti-de Sitter space of any dimension. First, we give a classification result with respect to their asymptotic boundary, namely we show that every admissible sphere $\Lambda$ is the boundary of a unique such hypersurface, for any given value $H$ of the mean curvature. We also demonstrate that, as $H$ varies in $\mathbb{R}$, these hypersurfaces analytically foliate the invisible domain of $\Lambda$. Finally, we extend Cheng-Yau Theorem to the Anti-de Sitter space, which establishes the completeness of any entire constant mean curvature hypersurface.
\end{abstract}

\maketitle
	
	\tableofcontents
	\section*{Introduction}	
	\subsection*{Hystorical background} 
Constant mean curvature (CMC) hypersurfaces are a classical object in differential geometry. In Lorentzian geometry, the study of CMC \emph{spacelike} hypersurfaces is motivated by general relativity (see \cite{expo,expo2} for recent surveys on the topic). Indeed, using a CMC spacelike hypersurface as initial data for Einstein equations, the associated Cauchy problem greatly simplifies: for this reason, several authors focused on existence, uniqueness and non-existence problems for CMC spacelike hypersurfaces in Lorentzian manifolds (see for example \cite{cho,ecker,bartnikH,bartnik}). Moreover, foliations by CMC spacelike hypersurfaces define natural time coordinates, useful to understand the global geometry of a Lorentzian manifold (see for example \cite{fol3,fol1,fol2}).

The constant sectional curvature cases have been largely studied: for the flat case, \textit{i.e.} the Minkowski space $\R^{n,1}$, see for example \cite{cheng-yau,trei,choi-trei,abbz,bss}; for the negatively curved space, \textit{i.e.} the Anti-de Sitter space $\hyp^{n,1}$, see for example \cite{bbz,abbz,tamb}. In geometric topology, the study of surfaces in $\R^{2,1}$ and $\hyp^{2,1}$ has become of great interest since the pioneering work of Mess \cite{mes}, mostly because of their relation with Teichm\"{u}ller theory (see \cite{bbz,benbon,bonsch,tamb,sep,bonsep}). In higher dimension, see \cite{abbz,barmer}.

Further progress have been made on specific pseudo-Riemannian spaces, notably on the so-called \textit{indefinite space-forms} of signature $(p,q)$: the pseudo-Euclidean space $\R^{p,q}$, the pseudo-hyperbolic space $\hyp^{p,q}$ and pseudo-spherical space $\sph^{p,q}$. In higher co-dimension, CMC hypersurfaces generalize to \emph{parallel} mean curvature spacelike $p-$submanifold. See the works \cite{ishi,kkn} for estimates on the geometry of parallel mean curvature spacelike hypersurfaces in this setting. Recently, maximal spacelike hypersurfaces in $\hyp^{p,q}$ have been studied in relation with the higher Teichm\"{u}ller theory (see \cite{dgk,ctt,ltw,lt,sst,beykas}).

	\subsection*{Main results} 
The Anti-de Sitter space $\hyp^{n,1}$ is a Lorentzian manifold with constant sectional curvature $-1$. It is a pseudo-Riemannian symmetric space associated to $\mathrm{O}(n,2)$ and it identifies with the space of oriented negative lines of a non-degenerate bilinear form of signature $(n,2)$. As for the hyperbolic space, it admits a conformal boundary $\pd\hyp^{n,1}$, consisting of oriented degenerate lines of such bilinear form, which is diffeomorphic to $\sph^{n-1}\times\sph^1$. 

For a fixed real number $H$ and a given subset $\Lambda$ in the asymptotic boundary of $\hyp^{n,1}$, the asymptotic $H-$Plateau's problem in $\hyp^{n,1}$ consists in finding the spacelike hypersurfaces with constant mean curvature equal to $H$, whose asymptotic boundary coincides with $\Lambda$. A hypersurface is spacelike if the induced metric is Riemannian, which gives a constraint on the subsets $\Lambda$ for which the asymptotic $H-$Plateau's problem can be \textit{a priori} solved: we call \textit{admissible} a subset $\Lambda\sq\pd\hyp^{n,1}$ satisfying such constraint. For a more intrinsic characterization of admissible boundaries, see Definition~\ref{de:adm}. Our first achievement is to solve the asymptotic $H-$Plateau's problem in $\hyp^{n,1}$: in particular, we prove that the admissibility condition is not only necessary but also sufficient.	

	\begin{thmx}\label{teo:A}
		For any $H\in\R$, any admissible boundary in $\pd\hyp^{n,1}$ bounds a unique spacelike hypersurface with constant mean curvature equal to $H$.
	\end{thmx}

Many partial results have been previously obtained. To our knowledge, the first progress has been made by Andersson, Barbot, Béguin and Zeghib \cite{bbz,abbz}, who solved the problem for $\Lambda$ invariant by the action of a subgroup $\Gamma$ of $\Isom(\hyp^{n,1})$ that acts freely and properly discountinuously on a properly embedded spacelike hypersurface $S$ of $\hyp^{n,1}$ with $S/\Gamma$ compact. Equivalently, these subgroups are known in the literature as \emph{convex cocompact}. The result of existence and uniqueness has been extended in the 3-dimensional case $\hyp^{2,1}$ by Bonsante and Schlenker for the maximal case, \textit{i.e.} $H=0$ (\cite{bonsch}), and by Tamburelli for arbitrary values of $H$ (\cite{tamb}), to the class of quasi-symmetric boundary curves. Moreover, \cite{bonsch} proved the existence of maximal hypersurfaces in any dimension, for $\Lambda$ the graph of a strictly $1-$Lipschitz map from $\sph^{n-1}\to\sph^1$. Recentely, Seppi, Smith and Toulisse solved the Plateau problem for spacelike maximal $p-$submanifold in $\hyp^{p,q}$, for all admissible boundaries (\cite{sst}). 

It is important to remark that all the above uniqueness results hold in the class of \emph{complete} submanifolds. This is a highly non-trivial requirement in Lorentzian geometry. Indeed, a properly embedded hypersurface in $\hyp^{n,1}$ might not be complete: if the normal vector degenerates fastly enough, the metric is incomplete. The same phenomenon occurs in $\R^{n,1}$ (see also \cite{bss2}).

The asymptotic $H-$Plateau problem has been studied in the flat case, \textit{i.e.} in the Minkowski space. In particular, the maximal case has been solved in \cite{cheng-yau}, while for $H\ne0$ several achievements have been made (\cite{trei,choi-trei}), until the quite recent complete classification (\cite{bss}). It is worth noticing that, in this setting, the CMC hypersurfaces are proved to be convex (\cite{trei}). This does not hold true in the Anti-de Sitter setting: not only for $H=0$, where a maximal hypersurface is convex if and only if it is totally geodesic, but there are examples of non-convex CMC hypersurfaces also for $H\ne0$ (see Remark~\ref{oss:nonconvex} and Remark~\ref{oss:convex}).

For $\R^{n,1}$, the uniqueness result holds in the class of properly embedded hypersurfaces, due to the remarkable result \cite[Corollary of Theorem~1]{cheng-yau}, which states that any properly embedded CMC hypersurface in Minkowski space is complete. We extend this result to the Anti-de Sitter case:

	\begin{thmx}\label{teo:B}
		Any properly embedded spacelike hypersurface with constant mean curvature in $\hyp^{n,1}$ is complete.
	\end{thmx}

In the literature, many results about CMC hypersurfaces in Anti-de Sitter space assume completeness. In particular, in Appendix~\ref{app:a} we focus on \cite{ishi,kkn}. 

Lastly, we study the behaviour of CMC hypersufaces with the same asymptotic boundary $\Lambda$, as their mean curvature $H$ varies in $\R$. Denote $\Omega(\Lambda)$ the union of entire spacelike hypersurfaces spanning $\Lambda$, \emph{i.e} $x\in\Omega(\Lambda)$ if there exists an entire spacelike hypersurface $S$ such that $x\in S$ and $\pd S=\Lambda$. It is known that $\Omega(\Lambda)$ is a globally hyperbolic geodesically convex open subset of $\hyp^{n,1}$, called \textit{invisible domain}.
	
	\begin{thmx}\label{teo:C}
		The invisible domain $\Omega(\Lambda)$ of an admissible boundary $\Lambda$ in $\pd\hyp^{n,1}$ is real-analytically foliated by complete CMC spacelike hypersurfaces spanning $\Lambda$.
	\end{thmx}

If the boundary $\Lambda$ is equivariant by the action of a convex cocompact subgroup of $\Isom(\hyp^{n,1})$, in the quotient we recover the same foliation described in \cite{bbz,abbz}, since the uniqueness of the CMC hypersurfaces guarantees their invariance by the action of any group preserving the boundary. If $n=2$ and $\Lambda$ is a quasi-symmetric boundary, we recover the foliation of \cite{tamb}, and we improve its regularity.

In the language of \cite{bbz,abbz}, any maximal globally hyperbolic Cauchy compact Anti-de Sitter manifold admits a unique CMC time function, which is real-analytic. A \textit{time function} on a Lorentzian manifold $M$ is a submersion $\tau\colon M\to\R$, strictly monotone on timelike curves. Moreover, it is a \textit{CMC} time function if $\tau^{-1}(H)$ is an $H-$hypersurface, for any $H\in\tau(M)$. Using the same language, Theorem~\ref{teo:C} can be rephrased as follows:

\begin{corx}\label{cor:D}
	Any maximal globally hyperbolic Cauchy complete Anti-de Sitter manifold admits a unique CMC-time function, which is real-analytic.
\end{corx}
	
	\subsection*{Main ingredients}
As previously mentioned, completeness is a highly non-trivial condition in the Lorentzian setting. Specifically, in the case of CMC hypersurfaces in Anti-de Sitter space, it has been proved to be equivalent to having bounded second fundamental form. Indeed, any properly embedded spacelike hypersurface with bounded second fundamental form is complete, as proved in \cite[Proposition~6.2.1]{benbon}, \cite[Corollary~3.30]{ltw} and \cite[Lemma~3.11]{sst}. Conversely, the main result of \cite{ishi,kkn} provides a universal bound on the second fundamental form of any properly embedded complete CMC hypersurface in $\hyp^{n,1}$, only depending on the mean curvature $H$. See Appendix~\ref{app:a} for details.\\

\noindent\textit{Bound on the second fundamental form.} Through geometrical arguments of a global nature, relying on the local estimates contained in \cite{bartnik,ecker}, we provide uniform bounds on the norm of the second fundamental form and its derivatives (Theorem~\ref{teo:boundII}). More specifically, we specialize these estimates to the Anti-de Sitter space, proving that, for CMC properly embedded hypersurfaces, there exists a bound only depending on the distance between the hypersurface and the boundary of its domain of dependence. Right after, we prove that such distance is uniformly bounded from below, \emph{i.e.} the bound is uniform. Our bounds are not explicit but, \textit{a posteriori}, we retrieve the bounds of \cite{ishi,kkn}, by applying Theorem~\ref{teo:B}, which is in fact a corollary of Theorem~\ref{teo:boundII}.\\

\noindent\textit{Uniqueness.} In the context of Anti-de Sitter and, more generally, pseudo-hyperbolic geometry, all the previous results of uniqueness of maximal and CMC submanifolds (\cite{bonsch,tamb,ltw,sst}) are proved in the class of \emph{complete} submanifolds. The proofs rely on an application of the Omori-Yau maximum principle, which indeed requires completeness. Hence, \textit{a priori}, there could exist several properly embedded $H-$hypersurfaces sharing the same boundary, among which only one complete. In light of Theorem~\ref{teo:B}, this is not possible, since there exists no incomplete properly embedded CMC hypersurface.

Here, instead of using Omori-Yau maximum principle, we prove uniqueness by a different method: we prove a maximum principles describing the mutual position of spacelike hypersurfaces, depending on their boundaries and their mean curvature (Theorem~\ref{teo:ord}). Roughly speaking, we show that the bigger the mean curvature, the more curved in the past is the hypersurface. More precisely, let $\Sigma_1$ and $\Sigma_2$ be two properly embedded hypersurfaces such that $\pd\Sigma_1$ lies in the past of $\pd\Sigma_2$. If $H_1\ge H_2$, then $\Sigma_1$ lies in the past of $\Sigma_2$.

The proof consists in using isometries and topological arguments to ensure that the maximum of the Lorentzian distance between $\Sigma_1$ and $\Sigma_2$ is reached in the interior of $\Sigma_1\times\Sigma_2$. In this case, we can apply a classical maximum principle, which does not require the completeness assumption.

This method has the advantage of, on the one hand, providing barriers (Proposition~\ref{lem:umbilical}) which are needed in the proof of the existence of CMC hypersurfaces. On the other hand, the result applies to a more general context, and could be useful in the study of hypersurfaces with prescribed mean curvature in the Anti-de Sitter space.\\

\noindent\textit{Existence and compactness.} An important ingredient is to show that the limit of a suitable sequence of spacelike hypersurfaces with constant mean curvature is a properly embedded spacelike CMC hypersurface. This result is achieved by the means of \textit{barriers}, \emph{i.e.} hypersurfaces that uniformly bound the geometry of elements of the sequence, whose existence is a consequence of Theorem~\ref{teo:ord}.

In particular, for the existence part of Theorem~\ref{teo:A}, we obtain the entire solution $\Sigma$ as the limit of compact solutions $\Sigma_k$, whose existence is guaranteed by a result of \cite{ecker}. 

Incidentally, the compactness result (Proposition~\ref{pro:compact}), allows us to describe the topology of the moduli space of CMC entire hypersurfaces of $\hyp^{n,1}$ (Corollary~\ref{cor:compact}).\\

\noindent\textit{Foliation.} The invisible domain of $\Lambda$ is topologically foliated by CMC hypersurfaces as a consequence of two results: the compactness result and the strong maximum principle (Proposition~\ref{lem:umbilical}), which is a special case of Theorem~\ref{teo:ord}.

To promote the regularity of the foliation, we describe the space of spacelike deformations of a single leaf $\Sigma_H$ as an open subset $A$ of the Banach space of regular functions on $\Sigma_H$. Indeed, a function $v\colon\Sigma_H\to\R$ identifies with its normal graph over $\Sigma_H$, namely \begin{equation*}
	S_v=\{\exp_x\left(v(x)N(x)\right), x\in\Sigma_H\}.\end{equation*}

We prove that the mean curvature, seen as a differential operator on $A$, is analytically invertible at $\Sigma_H$. As a consequence, we deduce that the foliation of CMC hypersurfaces is analytic.
	
	\subsection*{Organization of the paper}
From Section~\ref{sec:ads} to Section~\ref{sec:causal}, we provide the necessary background. In Section~\ref{sec:uni}, we prove the maximum principle and the uniqueness part of Theorem~\ref{teo:A}. In Section~\ref{sec:reg}, the bounds on the second fundamental form are given and Theorem~\ref{teo:B} follows. In Section~\ref{sec:ex}, we conclude the proof of Theorem~\ref{teo:A}. In Section~\ref{sec:comp}, we describe the topology of the moduli space of entire CMC hypersurfaces through a compactness result. In Section~\ref{sec:foliation}, we prove Theorem~\ref{teo:C} and Corollary~\ref{cor:D}. Finally, in Appendix~\ref{app:a}, we give a different proof of \cite[Theorem~1]{kkn}.

	\subsection*{Acknowledgements}
	This paper would have not been possible without the constant support of Francesco Bonsante and Andrea Seppi: I am deeply grateful for their continuous encouragement, the enlightening discussions and the numerous yet careful readings of this article. I would also like to thank Paolo Bignardi and Alex Moriani for related discussions.

	\section{Anti de-Sitter space}\label{sec:ads}	
Anti-de Sitter manifolds are the Lorentzian analogous of hyperbolic manifolds, \emph{i.e.} pseudo-Riemannian manifolds with signature $(n,1)$ and constant sectional curvature $-1$. In this section, two models of Anti-de Sitter geometry are presented, namely the quadric model $\hypd^{n,1}$ and the universal cover $\hypu^{n,1}$.

	\subsection{Lorentzian geometry}\label{sub:lorgeo}
A Lorentzian metric $g$ on a manifold $M$ of dimension $n+1$ is a symmetric non-degenerate $2-$tensor with signature $(n,1)$. We distinguish tangent vectors by their causal properties: a vector $v$ is \textit{timelike}, \textit{lightlike} or \textit{spacelike} if $g(v,v)$ is respectively negative, null or positive. Furthermore, a vector is called \textit{causal} if it is not spacelike. A curve $c\colon I\to M$ is called timelike (resp. lightlike, spacelike, causal), if its tangent vector $c'(t)$ is timelike (resp. lightlike, spacelike, causal) for all $t\in I$. 

\begin{de}
	Two points are \textit{time-related} (resp. \textit{light-related}, \textit{space-related}) if there exists a timelike (resp. lightlike, spacelike) geodesic joining them.
\end{de}

\begin{de}\label{de:space}
	A $C^1-$submanifold of $M$ is \textit{spacelike} if the induced metric is Riemannian.
\end{de}

A Lorentzian manifold is \textit{time-orientable} if the set of timelike vectors in $TM$ has two connected components, which will always be our case. A time-orientation is the choice of one of the connected component: vectors are called \textit{future-directed} if they belong to the chosen connected component, \textit{past-directed} otherwise. 

\begin{de}\label{de:cone}
	The \textit{cone} of a subset $X$ of $M$ is the set $I(X)$ of points of $M$ that can be joined to $X$ by a timelike curve. 
	
	Once a time-orientation is set, one can distinguish the \textit{future} cone $I^+(X)$ and the \textit{past} cone $I^-(X)$, containing respectively the points that can be reached from $X$ along a future-directed or past-directed timelike curve.
\end{de}

	\subsection{Quadric model}\label{sub:quadric}
The quadric model for Anti-de Sitter geometry is the equivalent of the hyperboloid model of hyperbolic space, \emph{i.e.}
\[\hypd^{n,1}:=\{x\in\R^{n,2}, \pr{x,x}=-1\},\]
$\R^{n,2}$ being the pseudo-Euclidean space of signature $(n,2)$, namely $\R^{n+2}$ endowed with the bilinear form
\begin{equation*}
	\pr{x,y}:=x_1y_1+\dots+x_n y_n-x_{n+1}y_{n+1}-x_{n+2}y_{n+2}.
\end{equation*}

As for the hyperbolic space, $\hypd^{n,1}$ admits a conformal boundary, which consists of the oriented isotropic lines for the bilinear form $\pr{\cdot,\cdot}$, and it is conformal to $\pd\hyp^n\times\sph^1$, endowed with the Lorentzian metric $g_{\sph^{n-1}}-g_{\sph^{1}}$. One can prove it directely or see it as a consequence of Lemma~\ref{lem:conformal}.

The tangent space $T_x\hypd^{n,1}$ identifies with $x^\perp=\{y\in\R^{n,2},\pr{x,y}=0\}$ and the restriction of the scalar product to $T\hypd^{n,1}$ is a time-orientable Lorentzian metric with constant sectional curvature $-1$. 

One can show that $\mathrm{O}(n,2)$ is the isometry group of $\hypd^{n,1}$, and that it is \textit{maximal}: any linear isometry from $T_x\hypd^{n,1}$ to $T_y\hypd^{n,1}$ is the tangent map of an isometry of $\hypd^{n,1}$ sending $x$ to $y$. In particular, the isometry group acts transitively. Totally geodesic $k-$submanifolds of $\hypd^{n,1}$ are precisely open subsets of the intersection between $\hypd^{n,1}$ and $(k+1)-$vector subspaces of $\R^{n,2}$. In the following, we will abusively refer to \textit{maximal} totally geodesic submanifold simply as totally geodesic submanifold, where maximality is to be intended in the sense of inclusion for connected submanifolds.

In particular, geodesics of $\hypd^{n,1}$ starting from $x$ are of the form
	\begin{equation}\label{eq:geod}
		\exp_x(tv)=\begin{cases}
			\cos(t)x+\sin(t)v & \text{if $\pr{v,v}=-1$};\\
			x+tv & \text{if $\pr{v,v}=0$};\\
			\cosh(t)x+\sinh(t)v & \text{if $\pr{v,v}=1$}.
		\end{cases}
	\end{equation}
Indeed, any curve $\gamma(t)$ satisfying one of the equations above is contained in $\hypd^{n,1}$, and an easy computation shows that $\gamma''(t)$, namely the covariant derivative of $\gamma$ with respect to the flat metric of $\R^{n,2}$, is proportional to $\gamma(t)$, \emph{i.e.} normal to $T_{\gamma(t)}\hypd^{n,1}$.

\begin{oss}\label{oss:P+-}
	For $x\in\hypd^{n,1}$, consider $T_x\hypd^{n,1}$ as the linear subspace $x^\perp\sq\R^{n,2}$. By definition, $\hypd^{n,1}\cap x^\perp$ is the set of unitary timelike vectors in $T_x\hypd^{n,1}$. By the above discussion, $\hypd^{n,1}\cap x^\perp$ is a totally geodesic hypersurface, and a direct calculation shows that it is isometric to two disjoint copies of the hyperbolic space $\hyp^n$. We will denote $P_+(x)$ and $P_-(x)$ these copies of $\hyp^n$: $P_+(x)$ is the set of \emph{future-directed} unitary timelike vectors and $P_-(x)$ is the set of \emph{past-directed} unitary timelike vectors in $T_x\hypd^{n,1}$, once a time-orientation is set.
	
	Equation~\eqref{eq:geod} shows that timelike geodesics starting at $x$ are periodic curves $\gamma\colon\R\to\hypd^{n,1}$ such that $\gamma(k\pi)=(-1)^k x$. Moreover, as
		\[\exp_x\left(\pm\frac{\pi}{2}v\right)=\pm v,\]
	future-directed timelike geodesics intersect orthogonally $P_+(x)$ at $t=\pi/2$ and $P_-(x)$ at $t=-\pi/2$ (see Figure~\ref{fig:prelated}).
\end{oss}

For the purposes of this article, the main advantage of this model is that its causal structure is encoded by the scalar product of $\R^{n,2}$. 

\begin{lem}\label{lem:causalprod}
	Let $p,q\in\hypd^{n,1}$, $p\ne\pm q$. There exists a geodesic joining $p$ and $q$ if and only if $\pr{p,q}<1$. Moreover, $p,q$ are
	\begin{enumerate}
		\item time-related $\iff-1<\pr{p,q}<1$;
		\item light-related $\iff\pr{p,q}=-1$;
		\item space-related $\iff\pr{p,q}<-1$.		
	\end{enumerate}
	
	Finally, $\pr{p,q}=1$ (resp. $\pr{p,q}>1$) if and only if $(p,-q)$ and $(-p,q)$ are light-related (resp. space-related).
\end{lem}

The proof follows directly from Equation~\eqref{eq:geod}. See Figure~\ref{fig:prelated} to visualize the set of time-related (resp. light-related, space-related) points to a point $p$.

\begin{figure}[h]
	\centering
	\begin{minipage}[c]{.33\textwidth}
		\centering
		\includegraphics{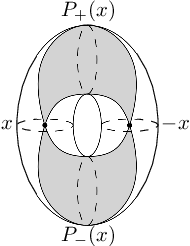}
	\end{minipage}%
	\begin{minipage}[c]{.33\textwidth}
		\centering
		\includegraphics{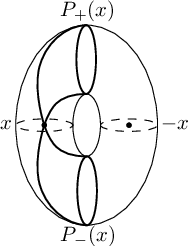}
	\end{minipage}%
	\begin{minipage}[c]{.33\textwidth}
		\centering
		\includegraphics{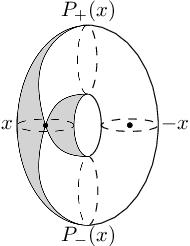}
	\end{minipage}%
	\caption{From the left to the right, the set of points time-related, light-related and space-related to $p$.}\label{fig:prelated}
\end{figure}

\begin{oss}\label{oss:Cp}
	Denote $C(-p)$ the cone of points light-related to $-p$. Lemma~\ref{lem:causalprod} implies the connected component of $\hypd^{n,1}\setminus C(-p)$ containing $p$ corresponds to the set of points satisfying $\pr{p,\cdot}<1$.
\end{oss}

	\subsection{The universal cover}
Let $P$ be a totally geodesic spacelike hypersurface of $\hypd^{n,1}$, and $p\in P$. Denote $N$ the unitary future-directed normal vector to $P$ at $p$, and define the map
\begin{equation}\label{eq:split}
	\begin{tikzcd}[row sep=1ex]
		\psi_{(p,P)}\colon P\times\R\arrow[r] &\hypd^{n,1}\\
		(x,t)\arrow[r,maps to] & R_t(x),
	\end{tikzcd}
\end{equation}
where $R_t$ is the linear map which is a rotation of angle $t$ restricted to $\Span(p,N)$ and fixes its orthogonal complement $\Span(p,N)^\perp$.

\begin{es}\label{es:psi}
	If we choose $p=(0,\dots,0,1,0)$ and $N=(0,0,\dots,0,1)$, namely \[P=\hypd^{n,1}\cap N^\perp=\hyp^n\times\{0\}\sq\R^{n,2},\] then $\psi_{(p,P)}(x,t)=(x_1,\dots,x_n,x_{n+1}\cos t,x_{n+1}\sin t)$.		
\end{es}

One can easily check that $\psi_{(p,P)}$ is a covering map whose domain is simply connected for any pair $(p,P)$. The pull-back metric can be explicitly computed:
\begin{equation*}\label{eq:pullback}
	\begin{split}
		g_{\hypu^{n,1}}:&=\psi_{(p,P)}^*g_{\hypd^{n,1}}=g_{P}+\pr{p,\cdot}^2dt^2\\
		&=g_{\hyp^n}-\cosh\left(\mathrm{d}_{\hyp^n}(p,\cdot)\right)^2dt^2,
	\end{split}
	\end{equation*}
where $g_P$ is the restriction of $g_{\hypd^{n,1}}$ to $P$, which is isometric to $\hyp^n$, and $\cosh\left(\mathrm{d}_{\hyp^n}(p,\cdot)\right)=\pr{p,\cdot}$ by Equation~\eqref{eq:geod}. Hence, the universal cover for the Anti-de Sitter space is $\hypu^{n,1}:=\hyp^n\times\R$, endowed with the metric $g_{\hypu^{n,1}}$.

\begin{de}\label{de:split}
	A \textit{splitting} of $\hypu^{n,1}$ is the choice of a pair $(p,P)$ identifing $\hypu^{n,1}$ with $\hyp^{n}\times\R$. We denote $x_0:=(0,\dots,0,1)\in\hyp^n$, namely $\{x_0\}\times\R=\psi_{(p,P)}^{-1}(\gamma)$, for $\gamma$ the timelike geodesic normal to $P$ at $p$.
\end{de}

By lifting the isometries of $\hypd^{n,1}$, it turns out that $\hypu^{n,1}$ has maximal isometry group, too. Moreover, since $\psi_{(p,P)}$ restricted to the slices $P\times\{t\}$ is linear, $P\times\{t\}$ is a totally geodesic spacelike hypersurface, for all $t\in\R$. In contrast, $\{x_0\}\times\R$ is the only fiber which is a (timelike) geodesic. 

We fix once and for all a time-orientation in the two models for $\AdS-$geometry: in the universal cover $\hypu^{n,1}$, we choose the one coinciding with the natural orientation of $\R$, and for the quadric model $\hypd^{n,1}$ the one induced by any covering map $\psi$ as in Equation~\eqref{eq:split}.

\begin{lem}\label{lem:conformal}
	Let $\sph_+^n\sq\sph^n$ be an open hemisphere, $\hypu^{n,1}$ is conformal to $\sph_+^n\times\R$ endowed with the metric $g_{\sph^n}-dt^2$, hence it has the same causal structure.
\end{lem}
\begin{proof}
	Let $f\colon\hyp^{n}\to\dis^n$ be an isometry between $\hyp^{n,1}$ and the Poincaré disk model, namely
	\[
	f_*g_{\hyp^n}=\left(\frac{2}{1-r^2}\right)^2 g_{\dis}.
	\]
	Explicitly, one can take $f$ to be the radial map such that, for $r\ge0$,
	\[
		r^2=\frac{x_{n+1}^2-1}{(1+x_{n+1})^2}.
	\]
	A direct computation shows that
	\[-\pr{x_0,x}=x_{n+1}=\frac{1+r^2}{1-r^2}.\]
	Using that the spherical metric on a hemisphere has the expression	
	\[g_{\sph^n}=\left(\frac{1-r^2}{1+r^2}\right)^2 g_{\dis}\]
	under the stereographic projection, we get
	\begin{equation*}\label{eq:poicaré}
		g_{\hypu^{n,1}}=g_{\hyp^n}-\pr{x_0,x}^2dt^2=\left(\frac{1+r^2}{1-r^2}\right)^2\left(g_{\sph^n}-dt^2\right).
	\end{equation*}
\end{proof}
\begin{oss}
	In particular, the conformal boundary $\pd\hyp^{n,1}$ identifies with $\pd\hyp^{n}\times\R$, endowed with the Lorentzian conformal structure of the metric $g_{\sph^{n-1}}-dt^2$.
\end{oss}

	\subsection{Fundamental regions}\label{sub:Up}
The lenght of a piecewise $C^1-$timelike curve $c\colon(a,b)\to\hypu^{n,1}$ is defined as
\[
\ell(c):=\int_a^b \sqrt{-g_{\hypu^{n,1}}\left(\dot{c}(t),\dot{c}(t)\right)}dt.
\]
\begin{de}
	Let $p\in\hypu^{n,1}$ and $q\in I(p)$. Their Lorentzian distance is
	\begin{equation*}%\label{eq:timedist}
		\dist(p,q):=\sup\{\ell(c),c\text{ timelike curve joining $p$ and $q$}\},
	\end{equation*} 
\end{de}
The Lorentzian distance satisfies the reverse triangle inequality
\begin{equation}\label{eq:triangleineq}
	\dist(p,q)\ge\dist(p,r)+\dist(r,q),
\end{equation}
provided that the three quantities are well defined and $r$ is chronologically between $p$ and $q$, that is either $r\in I^+(p)\cap I^-(q)$ or $r\in I^+(q)\cap I^-(p)$ (Definition~\ref{de:cone}). 

\begin{de}\label{de:Pp}
	For $p\in\hypu^{n,1}$, we denote $\mathcal{P}_+(p)$ (resp. $\mathcal{P}_-(p)$) the set at Lorentzian distance $\pi/2$ from $p$ in the future (resp. in the past). We will call it \textit{dual hypersurface} to $p$ in the future (resp. in the past) in $\hypu^{n,1}$. We also denote $p_\pm$ the unique point time-related to $p$ contained in $\{\dist(p,\cdot)=\pi\}\cap I^\pm(p)$.
\end{de}

\begin{oss}\label{oss:dist}
	One can prove that the distance between $p,q$ is achieved through a timelike geodesic if $q\in I^-(p_+)\cap I^+(p_-)$ (\cite[Corollary~2.13, Lemma~2.14]{bonsch}). This condition is not restrictive for our purposes, as we will show in Corollary~\ref{lem:invdom}.
\end{oss}

\begin{pro}\label{pro:distprod}
	Let $\psi\colon\hypu^{n,1}\to\hypd^{n,1}$ be a splitting. Then,
	\[
	\pr{\psi(p),\psi(q)}=-\cos\left(\dist(p,q)\right),
	\]
	for any pair $p,q$ of time-related points such that $q\in I^-(p_+)\cap I^+(p_-)$.
\end{pro}
\begin{proof}
	Consider two time-related points $p,q$ such that $q\in I^-(p_+)\cap I^+(p_-)$ and let  
	\[
	\gamma\colon[0,\dist(p,q)]\to\hypu^{n,1}
	\]
	be the timelike geodesic realizing the distance. Since $\psi(\gamma)$ is a timelike-geodesic, $\psi(p)$ and $\psi(q)$ are time-related, too. By Equation~\eqref{eq:geod},
	\[\psi\circ\gamma(s)=\cos(s)\psi(p)+\sin(s)v,\]
	for a unitary timelike tangent vector $v\in T_{\psi(p)}\hypd^{n,1}=\psi(p)^\perp$. By construction, $\psi(q)=\psi\left(\gamma(\dist(p,q))\right)$, hence
	\[\pr{\psi(p),\psi(q)}=-\cos\left(\dist(p,q)\right).\]
\end{proof}

\begin{cor}\label{cor:P}
	In any splitting, $\psi(\mathcal{P}_\pm(p))=P_\pm(\psi(p))$ and $\psi(p_\pm)=-\psi(p)$. In particular, $\mathcal{P}_\pm(p)$ is a totally geodesic spacelike hypersurface.
\end{cor}
\begin{proof}
	The function $\dist(p,\cdot)\colon I(p)\to\R$ is strictly monotone along timelike paths, and $\dist(p,p_\pm)=\pi$, hence 
	\[\mathcal{P}_+(p)\cup\mathcal{P}_-(\psi(p))\sq I^-(p_+)\cap I^+(p_-).\]
	Hence, we can apply Proposition~\ref{pro:distprod}: it follows that
	\[\psi\left(\mathcal{P}_+(p)\cup\mathcal{P}_-(p)\right)=\{\pr{\psi(p),\cdot}=0\}=P_+(\psi(p))\cup P_-(\psi(p)).\]
	In particular, since $\psi$ preserves the time-orientation, $\psi\left(\mathcal{P}_\pm(p)\right)=P_\pm(\psi(p))$.
	
	By construction, the dual hyperplane to $p_+$ in the past is $\mathcal{P}_+(p)$. The same argument proves that $\psi(\mathcal{P}_+(p))=P_-(\psi(p_+))$, hence $\psi(p_+)=-p$. The same applies to $p_-$, observing that its dual hyperplane in the future is $\mathcal{P}_-(p)$.
\end{proof}

\begin{de}\label{de:Up}
	For $p\in\hypu^{n,1}$ (resp. $x\in\hypd^{n,1}$) we define 
	\begin{align*}
		\mathcal{U}_p&:=I^-\left(\mathcal{P}_+(p)\right)\cap I^+\left(\mathcal{P}_-(p)\right)\\
		U_x&:=\{y\in\hypd^{n,1},\,\pr{x,y}<0\}
	\end{align*}
\end{de}
\begin{es}
	To visualise these object, see Figure~\ref{fig:Up}: one can check that in a splitting $(p,P)$, $\mathcal{P}_\pm(p)=\hyp^{n}\times\{\pm\pi/2\}$, $p_\pm=(p,\pm\pi)$ and $\mathcal{U}_p=\hyp^{n}\times(-\pi/2,\pi/2)$.
\end{es}

These sets are a bridge between the two models: indeed, $\mathcal{U}_p$ isometrically embeds in $\hypd^{n,1}$, and its image is precisely $U_{\psi(p)}$ (Corollary~\ref{cor:embU}). Moreover, any properly embedded spacelike hypersurface in $\hypu^{n,1}$ is contained in $\mathcal{U}_p$, for a suitable $p$ (Lemma~\ref{lem:osc}). Hence, the study of properly embedded spacelike hypersurfaces does not depend on the choice of the model of $\AdS-$geometry introduced so far (Remark~\ref{oss:spaceUp}).

\begin{cor}\label{cor:embU}
	Let $p\in\hypu^{n,1}$, $\psi$ be a splitting. The restriction of $\psi$ to $\mathcal{U}_p$ is an isometric embedding, whose image is $U_{\psi(p)}$.
\end{cor}
\begin{proof}
	Without loss of generality, let $p=(x_0,0)$ and $P=\hyp^{n}\times\{0\}$. By Corollary~\ref{cor:P}, $\mathcal{P}_\pm(p)=\hyp^{n}\times\{\pm\pi/2\}$, hence
	\[
	\mathcal{U}_p:=\left\{(x,t)\in\hypu^{n,1},-\frac{\pi}{2}<t<\frac{\pi}{2}\right\}.
	\]
	It follows that $\psi$ is injective on $\mathcal{U}_p$, since points in the same fiber as $(x,t)$ are of the form $(x,t+2\pi k)$, $k\in\Z$. Since $\psi$ is a local isometry and it is injective on the open set $\mathcal{U}_p$, the restriction $\psi|_{\mathcal{U}_p}$ is an isometric embedding. To conclude, by Example~\eqref{es:psi},
	\[\psi(\mathcal{U}_p)=\{x_{n+1}>0\}=\{\pr{\psi(p),\cdot}<0\}=U_{\psi(p)}.\]
\end{proof}
\begin{figure}[h]
	\begin{minipage}[c]{.45\textwidth}
		\centering
		\includegraphics{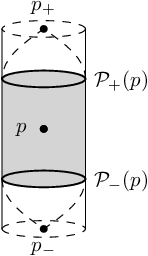}
	\end{minipage}%
	\begin{minipage}[c]{.45\textwidth}
		\centering
		\includegraphics{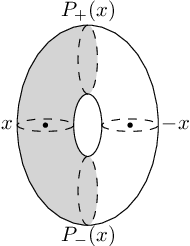}
	\end{minipage}%
	\caption{The fundamental regions $\mathcal{U}_p\sq\hypu^{n,1}$ in a splitting $(p,P)$ (left) and $U_p\sq\hypd^{n,1}$ (right).}\label{fig:Up}
\end{figure}

	\section{Graphs in Anti-de Sitter space}\label{sec:graphs}
All properly embedded spacelike hypersurfaces in $\hypu^{n,1}$ are graphs of functions $\hyp^n\to\R$ (Proposition~\ref{pro:emb}). Moreover, any properly embedded spacelike hypersurface is contained in a fundamental region (Remark~\ref{oss:spaceUp}): in other words, properly embedded spacelike hypersurfaces can be studied equivalently in each model of $\AdS-$geometry.

In the following, we consider $\hypu^{n,1}\cup\pd\hypu^{n,1}$ endowed with the conformal metric $g_{\sph^{n}}-dt^2$.

	\subsection{Achronal and acausal graphs}\label{sub:causalgraph}
The proofs of the following results can be found in \cite{bonsep}: even though they are stated for the 3-dimensional case, one can easily check that the arguments do not depend on the dimension.

\begin{de}\label{de:proj}
	A subset $X$ of $\hypu^{n,1}\cup\pd\hypu^{n,1}$ is \textit{achronal} (resp. \textit{acausal}) if no pair of points of $X$ can be joined by a timelike (resp. causal) curve of $\hypu^{n,1}\cup\pd\hypu^{n,1}$.
\end{de}

\begin{oss}\label{oss:Up}
	In other words, an achronal subset $X$ is contained in the complement of $I(p)$, for any $p\in X$. Since $\hypu^{n,1}\setminus I(p)\sq\mathcal{U}_p$, an achronal set $X$ of $\hypu^{n,1}$ is contained in the fundamental region $\mathcal{U}_p$, for any $p\in X$.
\end{oss}

	A splitting $(p,P)$ induces a map
	\[
	\begin{tikzcd}[row sep=1ex]
		\pi_{(p,P)}\colon\hypu^{n,1}\cup\pd\hypu^{n,1}=(\hyp^{n}\cup\pd\hyp^n)\times\R\arrow[r] & \hyp^{n}\cup\pd\hyp^n\\
		(x,t)\arrow[r,mapsto] & x
	\end{tikzcd},
	\]
	which is the projection on the first coordinate.

\begin{de}\label{de:graph}	
	An immersed subset $X\sq\hypu^{n,1}\cup\pd\hypu^{n,1}$ is a \emph{graph} if there exists a splitting $(p,P)$ such that the restriction of $\pi_{(p,P)}$ to $X$ is injective. Equivalently, $\pi_{(p,P)}$ restricted to $X$ admits an inverse $u\colon\pi_{(p,P)}(X)\to\R$ and $X=\gr u$. If there exists a splitting $(p,P)$ such that $\pi_{(p,P)}(X)=\hyp^n$, $X$ is called an \textit{entire} graph.
\end{de}

\begin{lem}{\cite[Lemma~4.1.2]{bonsep}}\label{lem:graph}
	For a subset $X$ of $\hypu^{n,1}\cup\pd\hypu^{n,1}$, the following statements are equivalent:
	\begin{enumerate}
		\item $X$ is achronal (resp. acausal);\label{it:1}
		\item there exists a splitting $(p,P)$ such that $X$ is the graph of a $1-$Lipschitz (resp. strictly $1-$Lipschitz) function;\label{it:2}
		\item for any splitting $(p,P)$, $X$ is the graph of a $1-$Lipschitz (resp. strictly $1-$Lipschitz) function.\label{it:3}
	\end{enumerate}
\end{lem}
In particular, an entire achronal hypersurface of $\hypu^{n,1}$ has a unique $1-$Lipschitz extension to the asymptotic boundary (\cite[Theorem~1]{mcs}).
\begin{pro}{\cite[Lemma~4.1.3]{bonsep}}\label{pro:emb}
	An achronal hypersurface $\Sigma$ in $\hypu^{n,1}$ is properly embedded if and only if it is an entire graph.
\end{pro}

\begin{de}
	For a bounded map $u\colon\hyp^n\to\R$, we define \[\osc(u):=\sup_{\hyp^n}u-\Inf_{\hyp^n}u.\]
\end{de}

\begin{lem}{\cite[Lemma~4.1.7]{bonsep}}\label{lem:osc}
	Let $\Sigma$ be an achronal properly embedded hypersurface in $\hypu^{n,1}$, and let $\Sigma=\gr u$ in a splitting. Then $\osc(u)\le\pi$, with equality if and only if $\Sigma$ is a totally geodesic degenerate hypersurface.
\end{lem}

Motivated by this result, we give the following definition:
\begin{de}\label{de:adm}
	A set $\Lambda\sq\pd\hypu^{n,1}$ is an \textit{admissible boundary} if it is the graph of a $1-$Lipschitz map $f\colon\pd\hyp^n\to\R$ and $\osc(f)<\pi$.
\end{de}

\begin{oss}
	Even though $\osc(f)$ is not invariant by isometries and depends on the splitting, Lemma~\ref{lem:osc} implies that the definition of admissible is intrinsic.
\end{oss}

\begin{oss}\label{oss:max}
	A direct consequence is the sharpness of Theorem~\ref{teo:A}: indeed, if a boundary is not admissible, it bounds no spacelike hypersurface. In particular, no CMC entire spacelike hypersurface.
\end{oss}

	\subsection{Spacelike graphs}\label{sub:graph}
As anticipated in Definition~\ref{de:space}, a $C^1-$embedded hypersurface $\Sigma$ is spacelike if the induced metric is Riemannian, or equivalently if the normal vector is timelike at every point.

We want to stress that being spacelike is a local property, while being achronal or acausal is a global one. In general, the notions are only partially related. Nonetheless, one can prove that any embedded spacelike hypersurface in $\hypu^{n,1}$ is locally the graph of a strictly $1-$Lipschitz function, but this is not true globally.

\begin{pro}{\cite[Lemma~4.1.5]{bonsep}}\label{pro:spaceacausal}
	An entire spacelike graph $\Sigma$ in $\hypu^{n,1}$ is acausal.
\end{pro}
\begin{proof}
	Let $\Sigma=\gr u$. By Lemma~\ref{lem:graph}, it suffices to prove that $u\colon\sph_+^n\to\R$ is strictly $1-$Lipschitz. The hypersurface $\Sigma$ can be described as the zeros of the function
	\[
	\begin{tikzcd}[row sep=1ex]
		\hyp^n\times\R\arrow[r] &\R\\
		(x,t)\arrow[r,mapsto] & u(x)-t
	\end{tikzcd}
	\] 
	
	Since being spacelike only depend on the conformal structure, we consider the normal space of $\Sigma$ with respect to the conformal metric $g_{\sph_+^{n}}-dt^2$. In this setting, $N\Sigma$ is generated by the gradient $\nabla^{\sph^{n}}u-\pd_t$, whose norm is equal to $|\nabla^{\sph^{n}}u|^2-1$. It follows that $\Sigma$ is spacelike if and only if $u$ is strictly $1-$Lipschitz.
\end{proof}	

\begin{oss}\label{oss:spaceUp}
	It follows by Proposition~\ref{pro:spaceacausal} and Remark~\ref{oss:Up} that any properly embedded spacelike hypersurface is contained in a fundamental region $\mathcal{U}_p$, for a suitable choice of $p\in\hypu^{n,1}$. In other words, there is a natural correspondence between properly embedded spacelike hypersurfaces in the two models.
\end{oss}

	\section{Causal structure}\label{sec:causal}
The causal structure of a Lorentzian manifold $M$ is the data of the causal vectors in the tangent space $TM$. First, we introduce the invisible domain, then the domain of dependence. It turns out that the domain of dependence of a properly embedded achronal hypersurface coincides with the invisible domain of its boundary (Proposition~\ref{cor:invdep}).

Finally, we define the convex hull of an admissible boundary, its past and its future part. In particular, we introduce the time functions (Proposition~\ref{pro:benbon}) that will provide barriers for Section~\ref{sec:ex}.

	\subsection{Invisible domain}\label{sub:invdom}

The invisible domain makes sense only on \emph{causal} Lorentzian manifold, \emph{i.e.} having no closed causal curve. Since $\hypd^{n,1}$ is not causal, we define the invisible domain only in the universal cover.

\begin{de}\label{de:invdom}
	The \textit{invisible domain} of an achronal subset $X$ of $\hypu^{n,1}\cup\pd\hypu^{n,1}$ is the set $\Omega(X)\sq\hypu^{n,1}$ of points that are connected to $X$ by no causal curve.
\end{de}

Roughly speaking, $\Omega(X)$ is the union of all acausal subset containing $X$. In light of Lemma~\ref{lem:graph}, let $X=\gr u$ in a splitting $(p,P)$: we can equivalently say that $\Omega(X)$ is the union of all acausal graphs of strictly $1-$Lipschitz extensions of $u$. In particular, one can consider the so called \textit{extremal extension} of $u$, denoted $u^-$ and $u^+$. They are extremal in the sense that for any $1-$Lipschitz extension $\tilde{u}$ of $u$, 
\[
u^-\le\tilde{u}\le u^+.
\]

Such extensions exist as they are respectively the supremum and the infimum of a set of bounded $1-$Lipschitz functions. One can prove that their graphs do not depend on the splitting: indeed, for two functions $f,g\colon\hyp^n\to\R$,
\[
f\le g \iff\gr f\sq I^-(\gr g),
\]
which does not depend on the splitting.

These ideas are summarized in the following two statements, whose proofs can be found in \cite[Lemma~4.2.2]{bonsep}.

\begin{lem}\label{lem:extremal}
	Let $X$ be an achronal subset of $\hypu^{n,1}\cup\pd\hypu^{n,1}$, $u^\pm$ its extremal extensions.
	\begin{enumerate}
		\item $\Omega(X)=I^+(\gr u^-)\cap I^-(\gr u^+)$;
		\item $\gr u^\pm$ is an achronal entire graph.
	\end{enumerate}
\end{lem}
\begin{es}\label{es:invdom}
	The invisible domain of a point $p$ is the set of space-related points to $p$, namely $\Omega(\{p\})=\hypu^{n,1}\setminus\overline{I(p)}$. The invisible domain of the boundary of the totally geodesic spacelike hypersurface $\mathcal{P}_+(p)$ is $I^-(p_+)\cap I^+(p)$ (see Figure~\ref{fig:invdom}).
\end{es}
\begin{figure}[h]
		\begin{minipage}[c]{.45\textwidth}
			\centering
			\includegraphics{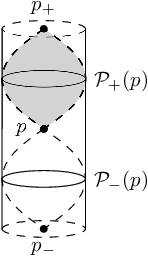}
		\end{minipage}%
		\begin{minipage}[c]{.45\textwidth}
			\centering
			\includegraphics{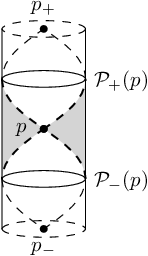}
		\end{minipage}%
\caption{Invisible domain of $\pd\mathcal{P}_+(p)$ (left) and $p$ (right).}\label{fig:invdom}
\end{figure}

\begin{lem}\label{lem:light}
	For $X$ an achronal closed subset of $\hypu^{n,1}\cup\pd\hypu^{n,1}$,
	\[\gr u^+\cap\gr u^-\]
	is the union of $X$ and all lightlike geodesic segments joining points of $X$.
\end{lem}

We are interested in acausal entire graphs, and we can get information through the study of the invisible domain of their asymptotic boundaries: for a graph in $\pd\hypu^{n,1}$, admissibility is a necessary condition to bound an achronal (or acausal) entire hypersurface (Lemma~\ref{lem:osc}). In fact, the converse is also true.

\begin{pro}\label{pro:admissible}
	For an achronal graph $\Lambda\sq\pd\hypu^{n,1}$, the following are equivalent:
	\begin{enumerate}
		\item $\Lambda$ is an admissible boundary;\label{adm:1}
		\item $\Omega(\Lambda)$ is not empty;\label{adm:2}
		\item there exists a properly embedded achronal hypersurface $\Sigma$ which is not a totally geodesic degenerate hypersurface such that $\pd\Sigma=\Lambda$.\label{adm:3}	\end{enumerate}
\end{pro}
\begin{proof}
	We first show $\eqref{adm:1}\implies\eqref{adm:2}$. If $\Omega(\Lambda)$ is empty, $u^+=u^-$ by Lemma~\ref{lem:extremal}. In particular, $\gr u^+$ contains a lightlike geodesic connecting two points of $\Lambda$ (Lemma~\ref{lem:light}). By direct computation, this implies that $\osc(u^+)=\pi$, \emph{i.e} $\Lambda$ is not admissible.
	
	The graph of the extremal extensions $u^+$ of $\Lambda$ is an achronal properly embedded hypersurface (Lemma~\ref{lem:extremal}), hence $\eqref{adm:2}\implies\eqref{adm:3}$: by contradiction, assume that $\gr u^+$ is a totally geodesic degenerate hypersurface, then any point of $\gr u^+$ is connected to $\Lambda$ by a lightlike geodesic. By Lemma~\ref{lem:light}, it follows that $\gr u^+=\gr u^-$, hence $\Omega(\Lambda)=\emptyset$ (Lemma~\ref{lem:extremal}), contradicting the assumption.
	
	To conclude, $\eqref{adm:3}\implies\eqref{adm:1}$ it is trivial: indeed, the boundary of a properly embedded achronal which is not a totally geodesic degenerate hypersurface is admissible by Lemma~\ref{lem:osc}. 
\end{proof}

	Finally, we show that the invisible domain of an admissible boundary isometrically embeds in $\hypd^{n,1}$, via any projection $\psi$ as in Equation~\eqref{eq:split}.
	
	\begin{lem}\label{lem:monoton}
		Let $\Lambda_1$, $\Lambda_2$ two achronal graphs in $\pd\hypu^{n,1}$. 
		\[
		\Lambda_1\sq I^-(\Lambda_2)\iff\pd_\pm\Omega(\Lambda_1)\sq I^-\left(\pd_\pm\Omega(\Lambda_2)\right).
		\]
	\end{lem}
	\begin{proof}
		By Lemma~\ref{lem:extremal}, the boundary of the invisible domain of $\Lambda_i=\gr u_i$ is the graph of its extremal extensions, namely \[\pd_\pm\Omega(\Lambda_i)=\gr u_i^\pm.\]
		
		In term of graphs, the statement is then equivalent to prove
		\[u_1<u_2\iff u_1^\pm<u_2^\pm.\]
		
		The implication $(\Leftarrow)$ is trivial, since $u_i$ is the restriction of $u_i^\pm$ to the boundary. Conversely, if $u_1<u_2$, then $v:=\max\{u_1^+,u_2^+\}$ is a $1-$Lipschitz function extending $u_2$. By definition of extremal extension $v\le u_2^+$, that is $u_1^+\le u_2^+$. To obtain the strict inequality, set \[\varepsilon:=\frac{1}{2}\min_{\pd\hyp^n}(u_2-u_1)>0\] and denote $u_\varepsilon:=u_1+\varepsilon$, which is still a $1-$Lipschitz map, strictly smaller than $u_2$. The same argument proves that
		\[
			u_1^+<u^+_\varepsilon\le u_2^+,
		\] 
		which concludes the proof.
	\end{proof}

	\begin{cor}\label{lem:invdom} 
			For an admissible boundary $\Lambda$ in $\pd\hypu^{n,1}$, \[\overline{\Omega(\Lambda)}\sq I^-(p_+)\cap I^+(p_-),\quad\forall p\in\Omega(\Lambda).\]
		\end{cor}
	\begin{proof}
		By definition, for any $X$ achronal subset, $X\sq\Omega(\Lambda)$ if and only if $\Lambda\sq\Omega(X)$.
		
		Take $X=\{p\}$, hence $\Omega(X)=\mathcal{U}_p\setminus\overline{I(p)}$. The frontier of $\mathcal{U}_p$ in $\hypu^{n,1}$ is $\pd\mathcal{P}_+(p)\cup\pd\mathcal{P}_-(p)$ (see Example~\ref{es:invdom}). We apply Lemma~\ref{lem:monoton} to get
			\[\overline{\Omega(\Lambda)}\sq I^-\left(\pd_+\Omega(\pd\mathcal{P}_+(p))\right)\cap I^+\left(\pd_-\Omega(\pd\mathcal{P}_-(p))\right).\]
			
		As mentioned in Example~\ref{es:invdom}, since $p,p_\pm$ are the dual points of $\mathcal{P}_\pm(p)$, we have
			\begin{align*}
					I^-\left(\pd_+\Omega(\pd\mathcal{P}_+(p))\right)&=I^-(p_+)\\
					I^+\left(\pd_-\Omega(\pd\mathcal{P}_-(p))\right)&=I^+(p_-)
				\end{align*}
		which concludes the proof.
		\end{proof}
	
	\begin{cor}\label{lem:bound:depdom}
		For any admissible boundary $\Lambda$ in $\pd\hypu^{n,1}$, $\overline{\Omega(\Lambda)}$ isometrically embeds in $\hyp^{n,1}$. Moreover, $\pr{\psi(p),\psi(q)}<1,\quad\forall p,q\in\Omega(\Lambda).$
	\end{cor}
	\begin{proof}
		For $p\in\hypu^{n,1}$, $I^-(p_+)\cap I^+(p_-)$ isometrically embeds in the connected component of $\hyp^{n,1}\setminus C(-\psi(p))$ containing $\psi(p)$ (see Figure~\ref{fig:emb}). We recall that this connected component is the set satisfying $\pr{\psi(p),\cdot}<1$ (Remark~\ref{oss:Cp}). 
		
		One concludes because $\overline{\Omega(\Lambda)}\sq I^-(p_+)\cap I^+(p_-)$, for all $p\in\Omega(\Lambda)$ (Corollary~\ref{lem:invdom}).
	\end{proof}
	\begin{figure}[h]
		\begin{minipage}[c]{.45\textwidth}
			\centering
			\includegraphics{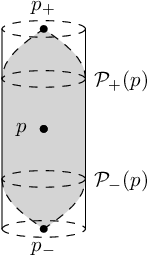}
		\end{minipage}
		\begin{minipage}[c]{.45\textwidth}
			\centering
			\includegraphics{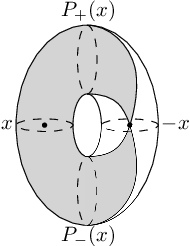}
		\end{minipage}
		\caption{$I^-(p_+)\cap I^+(p_-)$ isometrically embeds in $\{\pr{x,\cdot}<1\}$, for $x=\psi(p)$.}\label{fig:emb}
	\end{figure}
	
	\subsection{Domain of dependence}\label{sub:dep}
		Let $X$ be an acausal subset of a Lorentzian manifold $M$.	
		\begin{de}
			The \textit{domain of dependence} of $X$ is the set $D(X)$ of points $p\in M$ with the property that any inextensible causal path passing through $p$ intersects $X$.
		\end{de}
	\begin{oss}\label{oss:inext}
		Any inextensible causal path in $\hypu^{n,1}$ is properly embedded: indeed, in the same way as in Lemma~\ref{lem:graph}, one can check that a causal path is the graph of a $1-$Lipschitz map $f\colon(a,b)\sq\R\to\sph_+^n$. If the path is not properly embedded, without loss of generality we can assume that $a>-\infty$ and $f(t)\to x\in\sph^n_+$ for $t\to a^-$. By defining $f(t)=x$ for $t\le a$, we build a $1-$Lipschitz extension of $f$, hence $\gr f$ was not an inextensible causal path.
	\end{oss}
	
	\begin{de}\label{de:GH}
		A spacetime $(M,g)$ is called \textit{globally hyperbolic} if there exists an acausal subset $X$ such that $M=D(X)$. In this case, $X$ is called a \textit{Cauchy hypersurface} for $M$.
	\end{de}

	\begin{pro}\label{pro:depdomain}
		Let $\Sigma$ be an entire acausal graph in $\hypu^{n,1}$. A point $p\in\hypu^{n,1}$ belongs to $D(\Sigma)$ if and only if $I(p)\cap\Sigma$ is precompact in $\hypu^{n,1}$.
	\end{pro}
	\begin{proof}
		Without loss of generality, we assume $p$ is in the past of $\Sigma$.
		
		First, assume that $I(p)\cap\Sigma$ is precompact in $\hypu^{n,1}$. It follows that 
		\[\overline{I(p)}\cap\Sigma=\overline{I^+(p)}\cap\Sigma\] 
		is compact in $\hypu^{n,1}$, and so too is $\overline{I^+(p)}\cap\overline{I^-(\Sigma)}$. Hence, any future-directed causal curve starting at $p$ and not intersecting $\Sigma$, is contained in a compact set. Therefore it is not inextensible (Remark~\ref{oss:inext}). It follows that any inextensible future-directed causal curve starting at $p$ must intersect $\Sigma$, that is $p\in D(\Sigma)$.
		
		Conversely, if the intersection is not compact, there exists a point 
		\[
		q\in\overline{I^+(p)}\cap\pd\hypu^{n,1}\cap\overline{I^-(\pd\Sigma)}\ne\emptyset.
		\]
		Any inextensible causal line joining $p$ and $q$ does not meet $\Sigma$, hence $p\notin D(\Sigma)$.
	\end{proof}
	
	In fact, for an entire acausal graph, the domain of dependence only depends on its asymptotic boundary (see \cite[Corollary~3.8]{bonsch}, \cite[Proposition~4.4.6]{bonsep}):	
	\begin{pro}\label{cor:invdep}
		Let $\Sigma$ be an entire acausal graph, then $D(\Sigma)=\Omega(\pd\Sigma)$. In particular, two entire acausal graphs in $\hypu^{n,1}$ share the same domain of dependence if and only if they share the same boundary.
	\end{pro}
	
	\subsection{Convex hull}\label{sub:ch}
	A subset $\mathcal{C}$ of $\hypu^{n,1}\cup\pd\hypu^{n,1}$ is \textit{geodesically convex} if any pair of points in $\mathcal{C}$ is joined by at least one geodesic of $\hypu^{n,1}$ and any geodesic connecting them lies in $\mathcal{C}$. It follows that the intersection of geodesically convex sets is still geodesically convex.
	\begin{de}
		For a subset $X$ of $\hypu^{n,1}\cup\pd\hypu^{n,1}$ contained in a geodesically convex set $\mathcal{C}$, the \textit{convex hull} of $X$, denoted $\ch(X)$, is the smallest geodesically convex subset of $\hypu^{n,1}\cup\pd\hypu^{n,1}$ containing $X$.
	\end{de}
	\begin{oss}
		In general, a subset $X$ of $\hypu^{n,1}\cup\pd\hypu^{n,1}$ might admit no convex neighbourhood: indeed, $\hypu^{n,1}\cup\pd\hypu^{n,1}$ is not convex. To check that, take two points $p,q\in\hyp^{n,1}$ which can be connected by no geodesic, whose existence is due to Lemma~\ref{lem:causalprod}: there exists no geodesic connecting any pair of lifting $\tilde{p}$ and $\tilde{q}$ in $\hypu^{n,1}$.
	\end{oss}
	
	Nonetheless, the closure of the invisible domain of an admissible boundary is geodesically convex (\cite[Proposition~3.9]{bonsch}). It follows that, for $X$ an entire spacelike graph or an admissible boundary, the convex hull is well defined. Moreover, one can prove (\cite[Lemma~4.7]{bonsch}) that
	\[\ch(X)=\bigcap_{p\in X}\overline{\mathcal{U}_p}.\] 
	In particular, the convex hull is contained in a fundamental region, so it can be projected to $\hypd^{n,1}$. It turns out that the projection $\psi(\ch(X))\sq\hypd^{n,1}$ is the intersection of $\hypd^{n,1}$ and the convex hull of $X$ in $\R^{n,2}$ (see \cite[Section~4.6]{bonsep}).
	
	One can prove that $\Omega(\Lambda)$ is convex, for any admissible boundary $\Lambda$ (\cite[Proposition~4.6.1]{bonsep}). It follows that $\overline{\Omega(\Lambda)}$ is convex, which implies that $\ch(\Lambda)$ intersects the asymptotic boundary $\pd\hypu^{n,1}$ exactly in $\Lambda$: indeed, by minimality,
		\[
		\Lambda\sq\ch(\Lambda)\cap\pd\hypu^{n,1}\sq\overline{\Omega(\Lambda)}\cap\pd\hypu^{n,1}=\Lambda.
		\]
	\subsection{Past and future part}\label{sub:pastfut}
	We introduce two relevant functions on $\Omega(\Lambda)$ and state some of their main properties. 
	\begin{de}
		For $p\in\Omega(\Lambda)$, we denote $\tau_\past(p)$ the Lorentzian distance of $p$ from $\pd_-\Omega(\Lambda)$, that is
		\[\tau_\past(p):=\sup_{q\in\pd_-\Omega(\Lambda)\cap I^-(q)}\dist(p,q).\]
		Analogously, $\tau_\future$ stands for the Lorentzian distance from $\pd_+\Omega(\Lambda)$.
	\end{de}
	These functions are \textit{time functions}, namely
	\begin{de}\label{de:time}
		A \textit{time function} on a time-oriented Lorentzian manifold $(M,g)$ is a map $\tau\colon M\to\R$ strictly monotone along timelike paths.
	\end{de}
	\begin{oss}
		Usually, in the literature there is a distinction between the cases of strictly increasing and strictly decreasing functions (called \textit{reverse} time functions).
	\end{oss}
	
	Moreover, $\tau_\past$ and $\tau_\future$ have further remarkable properties, for which are known in the literature as cosmological times (see for example \cite{time}), when restricted respectively to	the past and the future of an admissible boundary.
	\begin{de}\label{de:pastfutpart}
	For an admissible boundary $\Lambda$, we define its \textit{past part} and its \textit{future part} to be 
		\begin{align*}%\label{eq:pastfutpart}
			\past(\Lambda)&:=I^-\left(\pd_+\ch(\Lambda)\right)\cap\Omega(\Lambda);\\
			\future(\Lambda)&:=I^+\left(\pd_-\ch(\Lambda)\right)\cap\Omega(\Lambda).
		\end{align*}
	\end{de}
	To visualize the past and the future of an admissible boundary, see Figure~\ref{fig:pastfut}.
	
	\begin{figure}[h]
		\begin{minipage}[c]{.23\textwidth}
			\centering
			\includegraphics{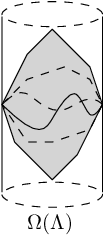}
		\end{minipage}
		\begin{minipage}[c]{.23\textwidth}
			\centering
			\includegraphics{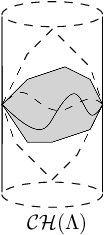}
		\end{minipage}
		\begin{minipage}[c]{.23\textwidth}
			\centering
			\includegraphics{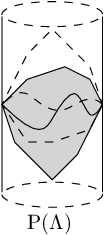}
		\end{minipage}
		\begin{minipage}[c]{.23\textwidth}
			\centering
			\includegraphics{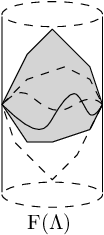}
		\end{minipage}
		\caption{From the left to the right, the invisible domain $\Omega(\Lambda)$, the convex core $\ch(\Lambda)$, the past part $\past(\Lambda)$ and the future part $\future(\Lambda)$ of an admissible boundary $\Lambda$.}\label{fig:pastfut}
	\end{figure}
	
	The following result has been proved in \cite[Proposition~6.19]{benbon} for the 3-dimensional case. However, the argument does not depend on the dimension, as already remarked in \cite{bonsch}.
	
	\begin{pro}\label{pro:benbon}	
		Let $\Lambda$ be an admissible boundary. Then $\tau_\past$ is a cosmological time for $\past(\Lambda)$, taking values in $(0,\pi/2)$. Specifically, for every point $p\in\past(\Lambda)$, there exist exactly two points $\rho^\past_-(p)\in\pd_-\Omega(\Lambda)$ and $\rho^\past_+(p)\in\pd_+\ch(\Lambda)$ such that:
		\begin{enumerate}
			\item $p$ belongs to the timelike segment joining $\rho^\past_-(p)$ and $\rho^\past_+(p)$;
			\item $\tau_\past(p)=\dist(\rho^\past_-(p),p)$;
			\item $\dist(\rho^\past_-(p),\rho^\past_+(p))=\pi/2$;
			\item $P(\rho^\past_\pm(p))$ is a support plane for $\past(\Lambda)$ passing through $\rho^\past_\mp(p)$;
			\item $\tau_\past$ is $C^1$ and $\nablah\tau_\past(p)$ is the unitary timelike tangent vector such that
			\[\exp_p\left(\tau_\past(p)\nablah\tau_\past(p)\right)=\rho^\past_-(p).\]
		\end{enumerate}
	for $\nablah\tau_\past$ the gradient of $\tau_\past$.
	\end{pro}
	\begin{oss}
		A symmetric result holds for $\tau_\future$ in $\future(\Lambda)$.
	\end{oss}
	
	\begin{cor}\label{cor:cosmo}
		Let $\Lambda$ be an admissible boundary. For any $p\in\Omega(\Lambda)$, 
		\[\tau_\future(p)+\tau_\past(p)\ge\frac{\pi}{2}.\]
	\end{cor}
\begin{figure}[h]
	\includegraphics{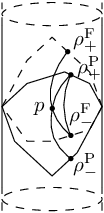}
	\caption{The two longest lines are the geodesic realizing the distance $\pi/2$ between the two pairs $\rho^\past_\pm(p)$ and $\rho^\future_\pm(p)$. The third one realizes the distance between $\rho^\past_+(p)$ and $\rho^\future_-(p)$.}\label{fig:dimtimesomma}
\end{figure}
	\begin{proof}
		It suffices to check the statement for $p\in\ch(\Lambda)$, that is $\past(\Lambda)\cap\future(\Lambda)$ (see Figure~\ref{fig:dimtimesomma}): indeed, if $p$ is not contained in $\past(\Lambda)$, Proposition~\ref{pro:benbon} ensures that $\tau_\past(p)\ge\pi/2$, and $\tau_\future$ is non-negative by construction. The same argument applies if $p\notin\future(\Lambda)$.
		
		By Proposition~\ref{pro:benbon}, it follows also that
		\begin{align*}
			\tau_\past(p)&=\frac{\pi}{2}-\dist(p,\rho^\past_+(p)),\\
			\tau_\future(p)&=\frac{\pi}{2}-\dist(p,\rho^\future_-(p)),
		\end{align*}
		where $\rho^\past_+(p)$ and $\rho^\future_-(p)$ are respectively the retractions on $\pd_+\ch(\Lambda)$ and $\pd_-\ch(\Lambda)$) induced by $\tau_\past$ and $\tau_\past$. We deduce that
		\begin{equation*}
			\tau_\past(p)+\tau_\future(p)=\pi-\dist(p,\rho^\past_+(p))-\dist(p,\rho^\future_-(p)).
		\end{equation*}
		The inverse triangle inequality (Equation~\eqref{eq:triangleineq}) concludes the proof:
		\begin{align*}
			&\dist\left(p,\rho^\past_+(p)\right)+\dist\left(\rho^\future_-(p),p\right)\le\dist\left(\rho^\future_-(p),\rho^\past_+(p)\right)\\
			&\le\dist\left(\rho^\future_-(p),\pd_+\ch(\Lambda)\right)=\frac{\pi}{2}-\tau_\past\left(\rho^\future_-(p)\right)\le\frac{\pi}{2}.
		\end{align*}
	\end{proof}
	
	The level sets of these time functions will be used as barriers in Section~\ref{sec:ex}.
	\begin{de}\label{de:w}
		For $\theta\in[0,\pi/2]$, we denote $\barrp{\theta}$ the hypersurface at Lorentzian distance $\theta$ from $\pd_+\ch(\Lambda)$, that is the level sets \[\left\{\tau_\past=\frac{\pi}{2}-\theta\right\}.\] 
		In particular, $\barrp0=\pd_+\ch(\Lambda)$, $\barrp{\pi/2}=\pd_-\Omega(\Lambda)$.
		
		Analogously, we denote $\barrf{\theta}$ the hypersurface at Lorentzian distance $\theta$ from $\pd_-\ch(\Lambda)$.
	\end{de}
	
	\begin{lem}
		Let $\Lambda$ be an admissible boundary, then $\barrp{\theta}$ and $\barrf{\theta}$ are spacelike Cauchy hypersurfaces for $\Omega(\Lambda)$, for any $\theta\in(0,\pi/2)$.
	\end{lem}
	\begin{proof}
		Without loss of generality, we fix $\theta\in(0,\pi/2)$ and prove the statement for $\barrp\theta$, which is the level set of a $C^1-$submersion. Its normal vector is $\nablah\tau_\past$, which is timelike by Proposition~\ref{pro:benbon}, hence $\barrp\theta$ is a spacelike hypersurface.
		
		To prove that it is a Cauchy hypersurface, take a point $p\in\Omega(\Lambda)$ and an inextensible future-directed causal path $c\colon(a,b)\to\Omega(\Lambda)$
		such that $c(0)=p$: we want to show that $c$ meets $\barrp\theta$.
		
		Being contained in $\Omega(\Lambda)$, $c$ is future-inextensible in $\hypu^{n,1}$ if and only if $c$ accumulates at $\Lambda$ (Remark~\ref{oss:inext}), which is impossible, since there exists no causal path connecting $p$ to $\Lambda$, by definition of invisible domain. Hence, up to reparameterization, we can assume $a,b\in\R$ are finite values, $c(a)\in\pd_-\Omega(\Lambda)$ and $c(b)\in\pd_+\Omega(\Lambda)$. 
		
		The boundary $\pd_+\ch(\Lambda)$ disconnects the future and the past boundary of $\Omega(\Lambda)$, hence there exists $t_+\in(a,b]$ such that $c(t_+)\in\pd_+\ch(\Lambda)$. It follows that \[\tau_\past\circ c\colon(a,t_+)\to(0,\pi/2)\]
		is a continuous function from a connected interval whose limit values are $0$ and $\pi/2$, hence it is surjective, namely $c(t)$ crosses $\barrp\theta$.
	\end{proof}
	
	\begin{cor}\label{cor:admspace}
		$\Lambda$ is an admissible boundary in $\hypu^{n,1}$ if and only if there exists a properly embedded spacelike hypersurface $\Sigma$ such that $\pd\Sigma=~\Lambda$.
	\end{cor}
	\begin{proof}
		The boundary of a properly embedded spacelike hypersurface is admissibile by Proposition~\ref{pro:spaceacausal}. Conversely, the level sets $W^\pm_\theta$ are properly embedded spacelike hypersurfaces with boundary $\Lambda$.
	\end{proof}
	
	\section{Maximum principles for mean curvature}\label{sec:uni}
	The main object of this article are properly embedded spacelike hypersurfaces with constant mean curvature (hereafter CMC). The second fundamental form of a spacelike $C^2-$hypersurface $\Sigma$ is the projection of the ambient Levi-Civita connection $\nablah$ on the future-directed normal space $N\Sigma$, which is the symmetric $(0,2)-$tensor on $T\Sigma$ defined by
	\begin{equation*}\label{de:IIf}
		\sff(v,w):=\pr{\nablah_v N,w},
	\end{equation*}
for $N$ the unitary future-directed vector field on $\Sigma$.
	
	\begin{de}\label{de:curvmedia}
		The \textit{mean curvature} $H$ of $\Sigma$ is the trace of the second fundamental form with respect to the induced metric. For an orthonormal frame $v_i$ of $T\Sigma$, 
		\begin{equation*}
			H:=\sum_{i=1}^n\sff(v_i,v_i).
		\end{equation*}
	\end{de}
	
	\begin{oss}
		The mean curvature is a function $H\colon\Sigma\to\R$, which is invariant under time-orientation preserving isometries of $\hypu^{n,1}$. Time-orientation reversing isometries change the sign of the mean curvature.
	\end{oss}	

	The first part of this section is devoted to prove a maximum principle (Theorem~\ref{teo:ord}), and to show that the uniqueness part of Theorem~\ref{teo:A} follows as a corollary of Theorem~\ref{teo:ord}. Finally, we give a stronger version of Theorem~\ref{teo:ord}, where one of the two hypersurfaces is a CMC (Proposition~\ref{lem:umbilical}), whose corollary (Corollary~\ref{cor:diam-barriera})  will play a key role in the proof of the existence part of Theorem~\ref{teo:A}, provided in Section~\ref{sec:ex}.
	
	\begin{oss}\label{oss:causalextended}
		Hereafter, a hypersurface $\Sigma$ will be a $C^2-$submanifold of co-dimension 1 without boundary. We denote $\pd\Sigma$ its topological frontier in $\hypu^{n,1}\cup\pd\hypu^{n,1}$. Moreover, we will consider the causal structure extended to the boundary, \emph{i.e} for $X\sq\hypu^{n,1}\cup\pd\hypu^{n,1}$, $I(X)$ refers to cone of $X$ in $\hypu^{n,1}\cup\pd\hypu^{n,1}$.
	\end{oss}
	
	\subsection{Weak maximum principle}
	Hereafter, for a hypersurface $\Sigma$, we will denote $\nabla$ the intrinsic Levi-Civita connection and $\nablah$ the exterior Levi-Civita connection, namely the connection of $\hypu^{n,1}$ and $\hypd^{n,1}$.
	The main result of this section is the following maximum principle:
	\begin{teo}\label{teo:ord}
		Let $\Sigma_1$ be a spacelike graph and $\Sigma_2$ be an entire spacelike graph with mean curvature respectively $H_1,H_2$.
		\begin{enumerate}
			\item If $H_1(p)\ge H_2(q)$, for every pair $(p,q)\in\Sigma_1\times\Sigma_2$ of time-related points, and $\pd\Sigma_1\sq\overline{I^-(\Sigma_2)}$, then $\Sigma_1\sq\overline{I^-(\Sigma_2)}$.
			\item If $H_1(p)\le H_2(q)$, for every pair $(p,q)\in\Sigma_1\times\Sigma_2$ of time-related points, and $\pd\Sigma_1\sq\overline{I^+(\Sigma_2)}$, then $\Sigma_1\sq\overline{I^+(\Sigma_2)}$.
		\end{enumerate}
	\end{teo}
	
First, let us show how uniqueness follows from Theorem~\ref{teo:ord}.
	\begin{repthmx}{teo:A}[Uniqueness]
		Let $\Sigma_1$, $\Sigma_2$ be two entire spacelike graphs in $\hypu^{n,1}$ sharing the same boundary and having the same constant mean curvature. Then $\Sigma_1=\Sigma_2$.
	\end{repthmx}
	\begin{proof}
	The pair $(\Sigma_1,\Sigma_2)$ satisfies the hypotheses of both items of Theorem~\ref{teo:ord}, hence 
		\[\Sigma_1\sq\overline{I^-(\Sigma_2)}\cap\overline{I^+(\Sigma_2)}=\Sigma_2.\]
		
		The other inclusion can be obtained by symmetry or observing that for entire graphs inclusion is equivalent to equality.
	\end{proof}
	
To demonstrate Theorem~\ref{teo:ord}, we will apply the maximum principle to the distance between the hypersurfaces, as in \cite{bonsch} and \cite{ltw}. In addition, a topological argument is used to maximize the distance over a compact set, in order to avoid the completeness hypothesis required in the cited approaches. 

\begin{proof}[Proof of Theorem~\ref{teo:ord}]
We focus on the first part of Theorem~\ref{teo:ord}: the second one follows because the map $\phi(x,t)=(x,-t)$ is an isometry of $\hypu^{n,1}$ which reverses the time-orientation, hence the sign of the mean curvature. 

\step{Stronger assumption}
We prove the statement assuming that $\pd\Sigma_1\sq I^-(\Sigma_2)$, instead of $\pd\Sigma_1\sq\overline{I^-(\Sigma_2)}$. The general statement follows directly by a continuity argument: indeed, if $\pd\Sigma_1\sq\overline{I^-(\Sigma_2)}$, it suffices to fix a splitting where $\Sigma_2=\gr u_2$, to apply the argument to $\Sigma_2(\delta)=\gr(u_2+\delta)$, $\delta>0$ and to take the limit as $\delta\to0$.

The statement reduces to prove that 
\begin{equation*}
	\mathcal{A}:=\{(p,q)\in\Sigma_1\times\Sigma_2|\, p\in I^+(q)\}=\emptyset.
\end{equation*}		

\step{$\mathcal{A}$ is precompact in $\Sigma_1\times\Sigma_2$}
We claim that the projection $\mathcal{A}_i$ of $\mathcal{A}$ over $\Sigma_i$ is precompact $\Sigma_i$. If so, $\mathcal{A}\sq\mathcal{A}_1\times\mathcal{A}_2$ is precompact in $\Sigma_1\times\Sigma_2$.
		
By definition, $\mathcal{A}_1=\Sigma_1\cap I^+(\Sigma_2)$ and by assumption $\pd\Sigma_1\sq I^-(\Sigma_2)$. Hence, $(\Sigma_1\cup\pd\Sigma_1)\cap I^-(\Sigma_2)$ is an open neighbourhood of $\pd\Sigma_1$ in $\Sigma_1\cup\pd\Sigma_1$. It follows $\mathcal{A}_1$ is precompact in $\Sigma_1$. A symmetric argument does not apply directly: $\Sigma_1$ might not be entire, and in that case $I(\Sigma_1)\ne\hypu^{n,1}\setminus\Sigma_1$, hence we can not state that $\pd\Sigma_2\sq I^+(\Sigma_2)$. Denote $S_2:=\Sigma_2\cap I(\Sigma_1)$:
 \[\mathcal{A}_2=\Sigma_2\cap I^-(\Sigma_1)=S_2\cap I^-(\Sigma_1),\] and $\pd S_2\sq I^+(\Sigma_1)$: by the same argument, it follows that $\mathcal{A}_2$ is precompact in $S_2$, hence in $\Sigma_2$, which proves the claim (see Figure~\ref{fig:Acomp} to visualize the proof).

\begin{figure}[h]
	\includegraphics{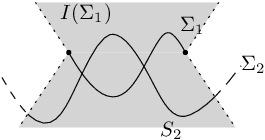}
	\caption{$\mathcal{A}$ is precompact in $\Sigma_1\times\Sigma_2$.}\label{fig:Acomp}
\end{figure}

\step{$\mathcal{A}$ is open in $\Sigma_1\times\Sigma_2$}
The fact that $\mathcal{A}$ is open follows directly from a more general result of Lorentzian geometry. Indeed, the same proof applies to any time-orientable Lorentzian manifold not containing closed causal curve, namely for any time-orientable \emph{causal} Lorentzian manifold. See Figure~\ref{fig:Aopen} to visualize the proof.

Fix $(x,y)\in\mathcal{A}$ and pick four points $a,b,c,d\in\hypu^{n,1}$ such that \[a<y<b<c<x<d,\]
where the order is given by the time-orientation, \emph{e.g.} $a<b$ means $a\in I^-(b)$, and the existence of such points $a,b,c,d$ is ensured by the hypothesis $y<x$.

Now, $V'_x:=I^+(c)\cap I^-(d)$ is an open neighbourhood of $x$ contained in $I^+(b)$, and $V'_y:=I^+(a)\cap I^-(b)$ is an open neighbourhood of $y$ contained in $I^-(c)$. In particular $V_x:=V'_x\cap\Sigma_1$ (resp. $V_y:=V'_x\cap\Sigma_2$) is an open neighbourhoods of $x$ in $\Sigma_1$ (resp. of $y$ in $\Sigma_2$). By transitivity of the order relation $<$, it follows that
\begin{align*}
	V_x\times \{w\}&\sq\mathcal{A},\quad \forall w\in V_y,\\
	\{z\}\times V_y&\sq\mathcal{A},\quad\forall z\in V_x,
\end{align*}
that is $V_x\times V_y\sq\mathcal{A}$ is an open neighbourhood of $(x,y)$.

\begin{figure}[h]
	\includegraphics{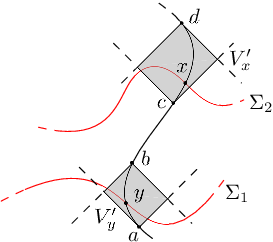}
	\caption{$\mathcal{A}$ is open in $\Sigma_1\times\Sigma_2$.}\label{fig:Aopen}
\end{figure}
\step{Project the problem in $\hyp^{n,1}$}
If $\Sigma_1\cap\Sigma_2=\emptyset$, the statement follows directly: indeed, $\Sigma_2\cup\pd\Sigma_2$ disconnects $\hypu^{n,1}\cup\pd\hypu^{n,1}$ by entireness, $\pd\Sigma_1\sq I^-(\Sigma_2)$ by hypothesis and $\Sigma_1$ lies in the same connected component as $\pd\Sigma_1$. Otherwise, $\Sigma_1$ and $\Sigma_2$ are contained in the fundamental region $\mathcal{U}_p$, for $p\in\Sigma_1\cap\Sigma_2$ (Remark~\ref{oss:Up}). Since $\mathcal{U}_p$ embeds in $\hypd^{n,1}$ (Corollary~\ref{cor:embU}), the distance can be computed through the scalar product (Proposition~\ref{pro:distprod}). In particular, we denote $A$ the image of $\mathcal{A}$ through the embedding $\mathcal{U}_p\hookrightarrow\hypd^{n,1}$.

For the rest of the proof, $F$ will denote the scalar product, namely $F(p,q):=\pr{p,q}$. The restriction of $F$ to $A$ takes value in $(-1,1)$, because $A$ is composed by pairs of time-related points in $\hypd^{n,1}$ (Lemma~\ref{lem:causalprod}). By continuity, $F(\overline{A})\sq[-1,1]$, and $F(\pd{A})\sq\{\pm1\}$ since, again by Lemma~\ref{lem:causalprod}, $F(p,q)\in(-1,1)$ if and only if $p,q$ are time-related. Furthermore, $\pd\Sigma_1\sq I^-(\Sigma_2)$ implies that $\Sigma_1\cap\overline{I^+(\Sigma_2)}$ is contained in the domain of dependence of $\Sigma_2$. Since
	\[\overline{\mathcal{A}}\sq\left(\Sigma_1\cap I^+(\Sigma_2)\right)\times\Sigma_2,\]
Corollary~\ref{lem:bound:depdom} ensures that $F(p,q)<1$ on $\overline{A}$, namely $F|_{\pd A}\equiv-1$. It follows that $F(\overline{A})$ is a compact subset of $[-1,1)$ and, more precisely, $F(\overline{A})=[-1,\max_{\overline{A}}F]$. In particular, $A=\emptyset$ if and only if $\max_{\overline{A}}F=-1$.
	
\step{Maximum principle}
By contradiction, assume $\max_{\overline{A}}F>-1$, namely $F$ reaches its maximum at $(\bar p,\bar q)\in A$. We showed in Step~3 that $A$ is open in $\Sigma_1\times\Sigma_2$: a direct computation leads to
\begin{equation}\label{eq:diff}
	d_{(p,q)}F(v,w)\big|_A=\pr{p,w}+\pr{v,q}.
\end{equation}

At the maximum $(\bar p,\bar q)$, $dF$ vanishes: by Equation~\eqref{eq:diff}, 
\begin{align*}
	&\bar p\in(T_{\bar q}\Sigma_2)^{\perp}=\Span\left(\bar q, N_2(\bar q)\right)\\
	&\bar q\in(T_{\bar p}\Sigma_1)^{\perp}=\Span\left(\bar p, N_1(\bar p)\right).
\end{align*}

Since both $N_1$ and $N_2$ are future-directed and $\bar p\in I^{+}(\bar q)$, there exists $T>0$ such that
\begin{equation}\label{eq:pq}
	\begin{cases}
		\bar p=\cos(T)\bar q+\sin(T)N_{2}(\bar q)\\
		\bar q=\cos(T)\bar p-\sin(T)N_{1}(\bar p)
	\end{cases}
\end{equation}
Equation~\eqref{eq:pq} has two important consequences: first, since $T\ne0$, at the maximum the tangent spaces are identified. Indeed,
\begin{equation*}%\label{eq:tang}
	T_{\bar q}\Sigma_2=\Span\left(\bar q, N_2(\bar q)\right)^{\perp}=\Span\left(\bar p, N_1(\bar p)\right)^{\perp}=T_{\bar p}\Sigma_1.
\end{equation*}

Moreover, we have the following equation:
	\begin{equation}\label{eq:normpt}
			\pr{\bar q,N_1(\bar p)}=\sin T=-\pr{\bar p,N_2(\bar q)}.
	\end{equation}
	
	To compute $\hess F$, we follow \cite[Lemma~4.3]{ltw} and add the proof for completeness. If $\gamma_1$ is a geodesic of $\Sigma_1$ such that $\gamma_1(0)=p$, $\dot{\gamma_1}(0)=v$,
	\[\frac{d^2}{dt^2}\gamma_1(t)|_{t=0}=\sff_1(v,v)N_1(p)+\pr{v,v}p.\]
	
	The formula can be easily derived by comparing the covariant derivative of $\gamma_1$ in $\Sigma$, $\hypd^{n,1}$ and $\R^{n,2}$. The same applies to a geodesic of $\Sigma_2$ such that $\gamma_2(0)=q$, $\dot{\gamma_2}(0)=w$, hence
	\begin{align*}
		&\hess_{(p,q)}F\left((v,w),(v,w)\right)=\frac{d^2}{dt^2}\pr{\gamma_1(t),\gamma_2(t)}|_{t=0}\\
		&=(\pr{v,v}+\pr{w,w})F(p,q)+2\pr{v,w}+\sff_1(v,v)\pr{N_1(p),q}+\sff_2(w,w)\pr{N_2(q),p}.
	\end{align*}
	Fix an orthonormal basis $(v_1,\dots,v_n)$ of $T_{\bar p}\Sigma_1=T_{\bar q}\Sigma_2$. At the maximum $(\bar p,\bar q)$, the Hessian is semi-negative definite: for every $i=1,\dots,n$, it holds
	\begin{align*}
		0&\ge\hess_{(\bar p,\bar q)}F\left((v_i,v_i),(v_i,v_i)\right)\\
		&=2F(\bar p,\bar q)+2+\sff_1(v_i,v_i)\pr{N_1(\bar p),\bar q}+\sff_2(v_i,v_i)\pr{N_2(\bar q),\bar p}\\
		&=2F(\bar p,\bar q)+2+\left(\sff_1(v_i,v_i)-\sff_2(v_i,v_i)\right)\sin T,
	\end{align*}
	where the last equation is due to Equation~\eqref{eq:normpt}. Summing over $i=1,\dots,n$, one obtains
	\begin{align*}
		0&\ge 2nF(\bar p,\bar q)+2n+\underbrace{\left(H_1(\bar p)-H_2(\bar q)\right)}_{\ge0}\sin T\ge 2nF(\bar p,\bar q)+2n\ge0.
	\end{align*}
	It follows that $\max_A F=-1$, which is absurd and concludes the proof.
\end{proof}
	
	\subsection{Strong maximum principle}
	If $\Sigma_2$ has constant mean curvature, one can promote Theorem~\ref{teo:ord} to a strong maximum principle.

	\begin{pro}\label{lem:umbilical}
		Let $\Sigma_1$ be a spacelike graph, $\Sigma_2$ be an entire CMC spacelike graph with mean curvature respectively $H_1,H_2$.
		\begin{enumerate}
			\item If $H_1\ge H_2$, and $\pd\Sigma_1\sq\overline{I^-(\Sigma_2)}$, then either $\Sigma_1\sq I^-(\Sigma_2)$ or $\Sigma_1\sq\Sigma_2$.
			\item If $H_1\le H_2$, and $\pd\Sigma_1\sq\overline{I^+(\Sigma_2)}$, then either $\Sigma_1\sq I^+(\Sigma_2)$ or $\Sigma_1\sq\Sigma_2$.
		\end{enumerate}
	\end{pro}
	The proof follows by combining Theorem~\ref{teo:ord} with \cite[Theorem~1]{eschen}, that is
	\begin{teo}\label{teo:maxprin}
		Let $(M,g)$ be a Lorentzian manifold. Consider two disjoint open sets $\Omega_1$, $\Omega_2$ of $M$ with $C^2$ spacelike boundaries $\Sigma_1$, $\Sigma_2$, respectively. Assume that $\Omega_1\sq I^-(\Omega_2)$ and that $\Sigma_1\cap\Sigma_2\ne\emptyset$. If there exists a constant $c\in\mathbb\R$ such that 
		\[H_1\ge c\ge H_2,\]
		for $H_i$ the mean curvature of $\Sigma_i$, then $\Sigma_1=\Sigma_2$ and $H_1=H_2=c$.
	\end{teo}

	\begin{proof}[Proof of Proposition~\ref{lem:umbilical}]
		As always, we prove only the first part. First, by Theorem~\ref{teo:ord}, $\Sigma_1\sq\overline{I^-(\Sigma_2)}$.
		
		If $\Sigma_1$ does not intersect $\Sigma_2$, we are done. Otherwise, fix a splitting $\hypu^{n,1}=\hyp^n\times\R$ and denote $u_i$ the function whose graph is $\Sigma_i$. Let $\Omega$ be the domain of $u_1$, \emph{i.e.} the projection of $\Sigma_1$ to $\hyp^n$, and consider $M:=\Omega\times\R$, which is an open subset of $\hypu^{n,1}$, hence a Lorentzian manifold.
		
		The condition $\Sigma_1\sq\overline{I^-(\Sigma_2)}$ translates as $u_1\le u_2$ on $\Omega$, hence $\Omega_1:=\{(x,t)\in M, t<u_1(x)\}$ and $\Omega_2:=\{(x,t)\in M, t>u_2(x)\}$ satisfy the hypotheses of Theorem~\ref{teo:maxprin}: indeed, they are disjoint, their boundaries are $C^2$, spacelike and intersect by assumption. Moreover, for $c:=H_2$, the inequality on the mean curvatures is satisfied: it follows that $\Sigma_1=\Sigma_2\cap M$, that is $\Sigma_1\sq\Sigma_2$.
	\end{proof}
	
	\subsection{Equidistant hypersurfaces}\label{sub:equidistant}
	The first example of non-maximal CMC entire graphs we introduce are the ones sharing the boundary with a totally geodesic spacelike hypersurface.
	\begin{lem}\label{lem:equi}
		Let $\mathcal{P}$ be a totally geodesic spacelike hypersurface of $\hypu^{n,1}$, $\theta\in(0,\pi/2)$. The set
		\[\{p\in\hypu^{n,1}, \dist(p,P)=\theta\}\]
		is the disjoint union of two entire totally umbilical graphs $\mathcal{P}^{+}_\theta$ and $\mathcal{P}^{-}_\theta$, contained respectively in the future and in the past of $\mathcal{P}$, with constant mean curvature $H_{\pm}(\theta)=\mp n\tan\theta$.
	\end{lem}
	\begin{proof}
		By construction, $\mathcal{P}^\pm_\theta\sq I^\pm(\mathcal{P})$, and $\mathcal{P}$ separates $\mathcal{P}^+_\theta$ and $\mathcal{P}^-_\theta$.
		
		To compute the mean curvature, consider the dual past point of $\mathcal{P}$, namely the point $e\in\hypu^{n,1}$ such that $\mathcal{P}_+(e)=P$. Denote $P^\pm_\theta:=\psi(\mathcal{P}^\pm_\theta)$, by Proposition~\ref{pro:distprod}
		\[P^\pm_\theta=\left\{\cos(\theta)p\pm\sin(\theta)e, p\in P\right\}.\]
		It follows that $\pff(v,w)=\cos\theta\pr{v,w}$ and $\sff(v,w)=\mp\sin\theta\pr{v,w}$, hence $B(v)=\mp\tan\theta(v)$, concluding the proof.
		\end{proof}
	
	\subsection{Barriers} We extend the maximum principle to the leaves of the cosmological time functions $\tau_\past$ and $\tau_\future$.
	\begin{pro}\label{pro:bar}
		Let $\Sigma$ be an entire spacelike graph, $\theta\in[0,\pi/2)$ and $\barrp{\theta},\barrf{\theta}$ as in Definition~\ref{de:w}, for $\Lambda=\pd\Sigma$.
		\begin{enumerate}
			\item If $H\ge n\tan\theta$, then either $\Sigma\sq I^{-}\left(\barrp{\theta}\right)$ or $\Sigma=\mathcal{P}^-_\theta$, for a spacelike hyperplane $\mathcal{P}$;
			\item if $H\le-n\tan\theta$, then either $\Sigma\sq I^{+}\left(\barrf\theta\right)$ or $\Sigma=\mathcal{P}^+_\theta$, for a spacelike hyperplane $\mathcal{P}$.
		\end{enumerate}
	\end{pro}
	\begin{oss}
		If $\theta=0$, namely the mean curvature of $\Sigma$ is non-negative (resp. non-positive), $\Sigma$ is contained in $\past(\pd\Sigma)$ (resp. $\future(\pd\Sigma)$), introduced in Subsection~\ref{sub:pastfut}. In particular, if $\Sigma$ is maximal, \emph{i.e.} the mean curvature identically vanishes, we recover a well known fact (see for instance \cite[Lemma~4.1]{bonsch}), that is that $\Sigma\sq\ch(\pd\Sigma)$: indeed, $\past(\pd\Sigma)\cap\future(\pd\Sigma)=\ch(\pd\Sigma)$, by definition (see Figure~\ref{fig:pastfut}).
	\end{oss}
\begin{figure}[h]
	\includegraphics{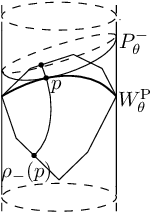}
	\caption{Any point in $\barrp\theta$ is contained in an equidistant hypersurface $\mathcal{P}_\theta^-$, for a suitable support hyperplane $\mathcal{P}$ of $\pd_+\ch(\Lambda)$.}\label{fig:bar}
\end{figure}
	\begin{proof}
		We focus on the first item, and we consider the problem in the quadric model $\hypd^{n,1}$.
		
		The idea of the proof is to find, for any point $p\in\barrp\theta$, a totally geodesic spacelike hypersurface $\mathcal{P}$ such that $\pd\mathcal{P}\sq I^+(\Sigma)$ and $p\in \mathcal{P}^-_\theta$, in order to apply the strong maximum principle (see Figure~\ref{fig:bar}).
		
		We recall that by Proposition~\ref{pro:benbon}, we can associates to $p$ the retraction $\rho^\past_-(p)\in\pd_-\Omega(\pd\Sigma)$. In particular, the future dual hyperplane $\mathcal{P}_+(\rho^\past_-(p))$ is a support hyperplane for $\pd_+\ch(\pd\Sigma)$ containing $\rho^\past_+(p)$. By construction, its boundary lies in the future of $\pd\Sigma$ and its distance from $p$ is 
		\[\dist\left(p,\rho^\past_+(p)\right)=\frac{\pi}{2}-\tau_\future(p)=\theta,\]
		that is $P(\rho^\past_-(p))_\theta^-$ contains $p$.
		
		Applying Proposition~\ref{lem:umbilical} to the two entire spacelike hypersurfaces $\Sigma$ and $\mathcal{P}(\rho^\past_-(p))_\theta^-$, we deduce that either $\Sigma$ does not contain $p$ or $\Sigma$ coincides with $\mathcal{P}(\rho^\past_-(p))_\theta^-$. Since $p$ was arbitrary, that concludes the proof.
	\end{proof}
	With the same argument, one can prove a local version of the proposition.
	\begin{cor}\label{cor:diam-barriera}
		Let $\Lambda$ be an admissible boundary, and $\theta\in[0,\pi/2)$. Let $\Sigma$ be a spacelike graph with mean curvature $H$.
		\begin{enumerate}
			\item If $\pd\Sigma\sq I^{-}\left(\barrp{\theta}\right)$ and $H\ge n\tan\theta$, then $\Sigma\sq I^{-}\left(\barrp{\theta}\right)$;
			\item if $\pd\Sigma\sq I^{+}\left(\barrf{\theta}\right)$ and $H\le-n\tan\theta$, then $\Sigma\sq I^{+}\left(\barrf{\theta}\right)$;
		\end{enumerate}
	\end{cor}
	
	\section{Estimates and completeness}\label{sec:reg}
	This section presents estimates on the gradient function, the norm of the second fundamental form and its derivatives. The results rely on more general estimates contained in \cite{bartnik} and \cite{ecker}. 
	
	First, Proposition~\ref{pro:locest} furnishes a local estimate for CMC spacelike graphs, crucial in Subsection~\ref{sec:ex}, to prove a compactness result for CMC (not necessarily entire) graphs.
	
	Theorem~\ref{teo:boundII} provides a global estimate for entire CMC spacelike graphs. This result is probably the most important contained in this article: indeed, the uniform bound on the second fundamental form implies completeness. On the other hand, the uniform bounds on the norm $\sff$ and $\nabla\sff$ play a key role in Section~\ref{sec:foliation}, to establish a uniform Schauder-type inequality (Proposition~\ref{pro:invert}).
	
	We denote $|\cdot|$ the norm induced by the Lorentzian metric. In particular, for a spacelike hypersurface $\Sigma$,
	\[|\sff|^2(x)=\sum_{i,j=1}^n\sff(v_i,v_j)^2,\]
	where $v_1,\dots,v_n$ forms an orthonormal basis of $T_x\Sigma$.
		
	First, some preliminary definitions: a unitary future-directed timelike $C^2(M)-$vector field $T\in\Gamma(TM)$ induces a Riemannian metric on $M$, given by
	\begin{align*}%\label{eq:riem}
		g_E(v,w):=g(v,w)+2g(T,v)g(T,w).	
	\end{align*}
	This reference metric is used to measure the size of tensors: for a tensor field $\Phi$ on $M$, we denote
	\begin{itemize}
		\item $\|\Phi\|:=\sup_{x\in M}g_E(\Phi_x,\Phi_x)^{1/2}$;
		\item $\|\Phi\|_{k}:=\sum_{i=0}^{k}\|\nablah^{i}\Phi\|=\|\Phi\|+\|\nablah\Phi\|+\|\nablah^2\Phi\|+\dots+\|\nablah^k\Phi\|$.
	\end{itemize}
		Here, $\nablah$ is the Levi-Civita connection of $(M,g)$ and we abusively denote $g_E$ the extension of the metric to the space of tensors on $M$.
		
		To observe how fast the metric of a spacelike hypersurface $\Sigma$ tends to a degenerate one, it is common to fix a suitable reference unitary future-directed timelike vector field $T$ and study the angle between $T$ and the normal vector field $N\Sigma$ via the so called \textit{gradient function}, namely
	\begin{equation}\label{eq:gradfun}
		\nu_\Sigma(p):=-g(N_p,T_p),
	\end{equation} 
	where $N$ denotes the unitary future-directed vector field normal to $\Sigma$. Hereafter, once a splitting is fixed, we take as $T$ the future-directed unitary vector field spanning $\pd_t$, \emph{i.e.}

	\begin{equation}\label{eq:T}
		T:=-\frac{1}{\sqrt{-g(\pd_t,\pd_t)}}\pd_t\in\Gamma(T\hypu^{n,1})=\Gamma(T\hyp^n\times T\R).
	\end{equation}

	Moreover, for a point $p\in\hypu^{n,1}$, we recall that $p_+$ is the unique point whose dual past hyperplane is $\mathcal{P}_+(p)$ (see Definition~\ref{de:Pp}).
	
	\subsection{Local estimates}
	The following estimate generalizes \cite[Lemma~4.13]{bonsch}, where the statement is given only for maximal hypersurfaces.
	
	\begin{pro}\label{pro:locest}
		For any $p\in\hypu^{n,1}$, $\varepsilon>0$, $L\ge0$ and $K\sq I^{-}(p_+)$ compact set, there exist constants $C_m=C_m(p,\varepsilon,L,K)$, $m\ge-1$, such that
		\begin{align*}
			&\sup_{\Sigma\cap I_\varepsilon^+(p)\cap K}\nu_\Sigma\le C_{-1};\\
			&\sup_{\Sigma\cap I_\varepsilon^+(p)\cap K}|\nabla^{m}\sff|^2\le C_m, m\ge0;
		\end{align*}
		for any spacelike graph $\Sigma$ with constant mean curvature $H\in[-L,L]$ such that $p\in D(\Sigma)$.
	\end{pro}
	\begin{figure}[h]
		\includegraphics{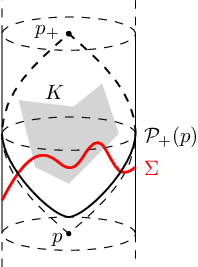}
		\caption{Setting of Proposition~\ref{pro:locest}.}\label{fig:loc}
	\end{figure}
	
	The proof of the first item relies on \cite[Theorem~3.1]{bartnik}. In the original result, the hypersurfaces are required to satisfy the so called \textit{mean curvature structure condition}, \textit{i.e.} there exists a constant $L$ such that
	\begin{equation}\label{mcsc}
		\begin{cases}
			|H_\Sigma|\le L\nu_\Sigma\\
			|\nabla H_\Sigma|\le L\left(\nu_\Sigma^2+\nu_\Sigma|\sff|\right)
		\end{cases}
	\end{equation}
	where $\nu_\Sigma$ is the gradient function associated to a spacelike hypersurface, as defined in Equation~\eqref{eq:gradfun}. In our setting, $H_\Sigma$ is constant and $\nu_\Sigma\ge1$, so Equation~\eqref{mcsc} is satisfied by any $L\ge|H_\Sigma|$. Hence, the result can be stated as
	\begin{teo}[\cite{bartnik}]\label{teo:bart}
		Let $\tau\in C^2(\hypu^{n,1})$ a time function in the region $\{\tau\ge0\}$. Assume there exist constants $c_0,c_1,c_2,c_3>0$ such that on $\{\tau\ge0\}$ it holds
		\begin{align*}
			&\langle\nablah\tau,\nablah\tau\rangle\le-c_0^{-2}, &\|\tau\|_{2}\le c_1,\\
			&\|T\|_2\le c_2, &\|\mathrm{Ric}\|\le c_3.
		\end{align*}
		$\forall L,\varepsilon>0$, there exists a constant $C=C(L,\varepsilon^{-1},c_0,c_1,c_2,c_3,\tau_{\max})$ such that
		\[\sup_{\Sigma\cap\{\tau\ge\varepsilon\}}\nu_\Sigma\le C,\]	
		for any $\Sigma\sq\hypu^{n,1}$ entire spacelike graph in $\hypu^{n,1}$, with constant mean curvature $H\in[-L,L]$ such that $\Sigma\cap\{\tau\ge0\}$ is compact and $\pd\Sigma\cap\{\tau>0\}=\emptyset$.
	\end{teo}
	
	The second item follows from \cite[Theorem~2.2]{ecker}. The original setting is the analysis of the mean curvature evolution in Lorentzian manifold satisfying the so called \emph{timelike convergence condition}, namely
	\begin{equation*}
		\mathrm{Ric}(X,X)\ge0 \quad\text{for any timelike vector field $X\in\Gamma(TM)$}.
	\end{equation*}
	A direct computation shows that any $\AdS-$manifold fulfills this condition.
	
	Given a spacelike immersion $F_{0}\colon\Sigma\to M$ in a Lorentzian manifold and a function
	\[\h\colon M\times[0,t_0]\to\R\]
	such that $X(\h)\ge0$ for any future-directed timelike vector field $X$, one considers the solution $F\colon M\times[0,t_0]\to M$ of the prescribed mean curvature flow
	
	\begin{equation}\label{MCF}\tag{$\text{MCF}_{\mathcal{H}}$}
		\begin{cases}
			\frac{\pd F}{\pd t}(x,t)=\left[(H-\h)N\right](x,t), & \\
			F|_{\Sigma\times\{0\}}=F_0. &
		\end{cases}
	\end{equation}
	$H(x,t)$ and $N(x,t)$ being respectively the mean curvature and the future-directed normal vector of $\Sigma_t:=F_t(\Sigma)$ at $x$.
	
	We state \cite[Theorem~2.2]{ecker} in the stationary case, \emph{i.e.} for $\h$ a constant function: it follows that $X(\h)=0$ for any vector field, $F_t=F_0$ and $\Sigma_t=\Sigma$, for all $t\in\R$.
	
	\begin{teo}[\cite{ecker}]\label{teo:eck}
		Let $\tau\in C^{2}(\hypu^{n,1})$ be a time function in the region $\{\tau\ge0\}$ and assume there exist constants $c_0,c_1,c_2>0$ such that on $\{\tau\ge0\}$ it holds
		\begin{align*}
			&\langle\nablah\tau,\nablah\tau\rangle\le-c_0^{-2}, & \|\tau\|_{2}\le c_1,\\
			&\|T\|_2\le c_2.
		\end{align*}
		$\forall L,\varepsilon>0$, there exist constants
		\[
		C_m=C_m(L,n,\varepsilon^{-1},c_0,c_1,c_2,\|\mathrm{Riem}\|_{m+1}),\quad m\in\N,
		\]
		such that
		\[\sup_{\Sigma\cap\{\tau\ge\varepsilon\}}|\nabla^{m}\sff|^2\le C_m.\]	
		for any smooth spacelike hypersurface  $\Sigma$ in $\hypu^{n,1}$, with constant mean curvature $H\in[-L,L]$ such that $\Sigma\cap\{\tau\ge0\}$ is compact and $\pd\Sigma\cap\{\tau>0\}=\emptyset$. 
	\end{teo}
	
	\begin{oss}\label{oss:prolocest}
		The idea of Proposition~\ref{pro:locest} is to apply Theorem~\ref{teo:bart} and Theorem~\ref{teo:eck} to suitably chosen $(\tau,T)$, and use the geometry of $\hypu^{n,1}$, in order to obtain bounds that only depend on ($p,K$).
	\end{oss}
	\begin{de}\label{de:conoeps}
		For $\varepsilon>0$, define $I_\varepsilon(p):=\{q\in I(p),\,\dist(p,q)>\varepsilon\}$. As for the cone, we denote to $I^+_\varepsilon(p)$ its future component and $I^-_\varepsilon(p)$ the past one  (see Figure~\ref{fig:conoeps}).
	\end{de}
	\begin{figure}[h]
		\includegraphics{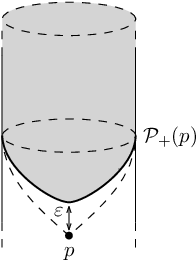}
		\caption{The future $\varepsilon-$cone $I^+_\varepsilon(p)$ at Lorentzian distance $\varepsilon$ from $p$.}\label{fig:conoeps}
	\end{figure}
	\begin{proof}[Proof of Proposition~\ref{pro:locest}]
		Fix $p\in\hypu^{n,1}$, $\varepsilon>0$, $L\ge0$ and $K\sq I^-(p_+)$. Let $\Sigma$ be a spacelike graph such that $p\in D(\Sigma)$. Denote $\tau:=\dist(p,\cdot)-\varepsilon/2$. By definition, $\{\tau>0\}=I_{\varepsilon/2}(p)$ and $\{\tau\ge0\}=\overline{I_{\varepsilon/2}(p)}$. Both sets are contained in $I(p)$, and $I(p)\cap\Sigma$ is precompact in $\hypu^{n,1}$ because $p\in D(\Sigma)$ (Proposition~\ref{pro:depdomain}). It follows that the function $\tau$ satisfies the conditions of Theorem~\ref{teo:bart} and Theorem~\ref{teo:eck}, namely
		\begin{equation*}
			\Sigma\cap\{\tau\ge0\}\ \text{is compact},\quad \pd\Sigma\cap\{\tau>0\}=\emptyset.
		\end{equation*}
		The region $\{\tau\ge0\}\cap K$ is contained in $I^{+}(p)\cap I^{-}(p_+)$, which we can isometrically embed in $\hypd^{n,1}$ so that
		\[
		\tau=\dist(p,\cdot)=\arccos(-\pr{p,\cdot})-\varepsilon/2\in C^{\infty}(I^+(p)).
		\] 
		Remark that the conditions on the constants $c_0,c_1,c_2$ of Theorem~\ref{teo:bart} and Theorem~\ref{teo:eck} coincide. Hence, it suffices to prove that $c_0,c_1,c_2,c_3$ only depend on $p,\varepsilon$ and $K$.
		
		Denoting $u(x)=\pr{p,x}$, $\nablah u(x)=p+\pr{p,x}x$. It follows
		\begin{align*}
			&\nablah\tau(x)=-\frac{1}{\sqrt{1+\pr{p,x}^2}}\nablah u=-\frac{p+\pr{p,x}x}{\sqrt{1+\pr{p,x}^2}}.
		\end{align*}
		Over $I^+(p)$, $\tau\ge0\iff\pr{p,x}^2\le\cos^2(\varepsilon/2)$, hence \[\pr{\nablah\tau(x),\nablah\tau(x)}=-\frac{1-\pr{p,x}^2}{1+\pr{p,x}^2}\le-\sin^2(\varepsilon/2)=:-c_0(\varepsilon)^{-2}.\]
		
		The compactness of $K$ allows to define
		\begin{align*}
			c_1:=\max_{K\cap\{\tau\ge0\}}\|\tau\|_2, && c_2:=\max_{K\cap\{\tau\ge0\}}\|T\|_2, && c_3:=\max_{K\cap\{\tau\ge0\}}\|\mathrm{Ric}\|_2.
		\end{align*}
		All constants only depend on $p,\varepsilon,K$ and meet the requirements of Theorem~\ref{teo:bart} and Theorem~\ref{teo:eck}, which concludes the proof.
	\end{proof}
	\begin{oss}
		$\hypu^{n,1}$ being a homogeneous space, one could think that $c_1,c_2,c_3$ in the proof are independent on the choice of $p$, and so are the constants $C_i$. This is not the case as $T$ is not invariant by isometries of $\hypu^{n,1}$, hence neither the induced Riemannian norm, so the constants $C_i$ do depend on the choice of $p$.
	\end{oss}
	
	\subsection{Global estimates}
We promote the local estimates of Proposition~\ref{pro:locest} to global bounds for the second fundamental form and its derivatives of any CMC entire spacelike graph in $\hypu^{n,1}$.

	\begin{teo}\label{teo:boundII}
		Let $L\ge0$, there exist constants $C_m(L,n)$, $m\in\N$, such that
		\begin{equation*}
			\sup_{\Sigma}|\nabla^{m}\sff|^2\le C_m(L,n),
		\end{equation*}
		for any entire spacelike graph $\Sigma$ in $\hypu^{n,1}$ with constant mean curvature $H\in[-L,L]$. 
	\end{teo}
\begin{figure}[h]
	\includegraphics{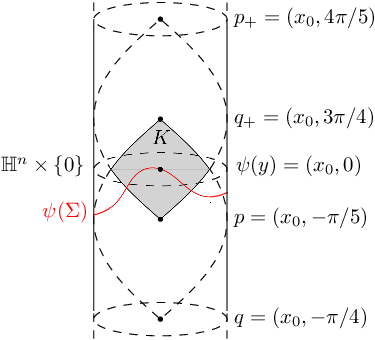}
	\caption{Setting of Proposition~\ref{teo:boundII}.}\label{fig:glob}
\end{figure}
	\begin{proof}
		In Proposition~\ref{pro:locest}, we showed that the constants $c_0,c_1,c_2$ only depend on the choice of $p$, $\varepsilon$ and $K$: here, we fix suitably $p$, $\varepsilon$ and $K$ and exploit the invariance the second fundamental form by the action of the isometry group, to have a global estimate.
		
		Fix a splitting $\hypu^{n,1}=\hyp^{n}\times\R$. Take an entire spacelike graph $\Sigma$ with constant mean curvature $H$ and fix a point $x\in\Sigma$. As remarked in Corollary~\ref{cor:cosmo}, \[\tau_\past(x)+\tau_\future(x)\ge\frac{\pi}{2}.\]
		It follows that there exists a timelike geodesic joining $x$ and $\pd\Omega(\pd\Sigma)$ whose lenght is at least $\pi/4$. In particular, there exists a point $y\in\Omega(\pd\Sigma)$ such that $\dist(x,y)=\pi/5$. Let $\phi\in\Isom\hypu^{n,1}$ such that $\phi(x)=(x_0,0)$, $\phi(y)=(x_0,-\pi/5)=:p$, whose existence is provided by maximality of $\Isom(\hypu^{n,1})$. (Observe that $\phi$ might not preserve the time-orientation, but in this case $\Sigma':=\phi(\Sigma)$ has constant mean curvature $-H$, which does not affect the result.)
		
		Denote $q:=(x_0,-\pi/4)$ and $K:=\overline{I^{+}(p)\cap I^{-}(q_+)}$, which is a compact set contained in $I^{-}(p_+)$ because $p\in I^{+}(q)$.
		
		By construction, $p\in\Omega(\pd\Sigma')=D(\Sigma')$, hence we can apply Proposition~\ref{pro:locest}, choosing $\varepsilon<\pi/5$, so that $\phi(x)=(x_0,0)\in I^+_\varepsilon(p)$: at $\phi(x)$, \[|\nabla^{m}\sff|^2\le C_m(L,p,\varepsilon,K,n).\] The second fundamental form is invariant by the action of $\Isom(\hypu^{n,1})$, $\Sigma$ and $x$ are arbitrary, $(p,\varepsilon,K)$ are fixed and that concludes the proof.
	\end{proof}

A first consequence of Theorem~\ref{teo:boundII} is 

\begin{repthmx}{teo:B}[Completeness]
	Any CMC entire spacelike graph in $\hypu^{n,1}$ is complete.
\end{repthmx}

Indeed, an entire graph with uniformly bounded second fundamental form in the Anti-de Sitter space is complete: see for example \cite[Proposition~6.21]{benbon}, \cite[Corollary~3.30]{ltw} or \cite[Lemma~3.11]{sst}.

As a second consequence, we deduce that any entire spacelike CMC graph in $\hypu^{n,1}$ has sectional curvature bounded both from below and above.
	\begin{cor}\label{cor:sectcurv}
		Let $L\ge0$. There exists a constant $K(L,n)$ such that for any entire spacelike graph $\Sigma$ with constant mean curvature $H\in[-L,L]$,
		\[|K_\Sigma(V)|\le K(L,n),\]
		for $K_\Sigma(V)$ the sectional curvature along a non-degenerate tangent 2-plane $V\sq T_p\Sigma$.
	\end{cor}
\begin{proof}
	Consider an orthonormal basis $v_1,v_2$ of $V$ and complete it to an orthonormal basis $v_i$ of $T_p\Sigma$. We recall that Gauss' equation is
	\[-1=K_{\hypu^{n,1}}(V)=K_{\Sigma}(V)-\sff(v_1,v_1)\sff(v_2,v_2)+\sff(v_1,v_2)^2.\]
	
	By Theorem~\ref{teo:boundII}, there exists a constant $C_0(|H|,n)$ such that $\sum h_{ij}^2\le C_0(|H|,n)$, $h_{ij}$ being the coefficients of $\sff$ in the orthonormal basis $v_i$. The proof follows by substituting in Gauss' equation the following estimates:
	\begin{align*}%\label{eq:stimaII}
		|\sff(v_l,v_k)|=|h_{lk}|\le1+h_{lk}^2\le\sum_{i,j=1}^n(1+h_{ij}^2)\le n^2+C_0(|H|,n).
	\end{align*}
\end{proof}
	
	\section{Existence}\label{sec:ex}
	In this section we will prove the first part of Theorem~\ref{teo:A}. The argument is a generalisation of the one used in \cite{bonsch}.
	\begin{repthmx}{teo:A}[Existence]
		Let $\Lambda$ be an admissible boundary in $\hypu^{n,1}$ and $H\in\R$. There exists a smooth entire spacelike graph $\Sigma$ with constant mean curvature $H$ and such that $\pd\Sigma=\Lambda$.	
	\end{repthmx}
	\begin{proof}
		\step{Equivalent statement}\label{step:ex1} We claim that it suffices to find an acausal Cauchy hypersurface $W$ for $\Omega(\Lambda)$ such that for any compact spacelike hypersurface $S$ with constant mean curvature $H$ and boundary contained in $W$, either $S\sq I^+(W)$ or $S\sq I^-(W)$.
		
		To prove the claim, we build a sequence of compact graphs with given constant mean curvature and prove that they smoothly converge to an entire graph, so that the limit has constant mean curvature.
		
		Fix a splitting $\hypu^{n,1}=\hyp^n\times\R$. For $r>0$, denote
		\begin{align*}
			B_r:=B_{\hyp^n}(0,r), && S_r:=W\cap(B_r\times\R).
		\end{align*}
		$S_r$ is the intersection between $W$ and the cylinder over $B_r$. By assumption, $W$ is an entire acausal graph, so $S_r$ is an acausal graph over $B_r$.
		
		Since $S_r$ is a compact acausal graph, by \cite[Theorem~4.1]{bartnik} there exists a compact smooth spacelike graphs $\Sigma_r$ with constant mean curvature $H$, such that $\pd\Sigma_r=\pd S_r$.
		
		By hypothesis, we can assume that infinitely many $\Sigma_r$ are contained in the same connected component of $\hypu^{n,1}\setminus W$. Without loss of generality, we assume $\Sigma_r\sq I^+(W)$. $\Sigma_r$ is the graph of a $1-$Lipschitz map $u_r$ defined on an open ball of $\sph_+^n$ (Lemma~\ref{lem:graph}). Moreover, $\Sigma_r$ is contained in $\Omega(\Lambda)$, then $|u_r|$ is uniformly bounded. By a diagonal process, we can extract a sequence $\Sigma_k=\gr u_k$ converging to an achronal entire graph $\Sigma=\gr u\sq\overline{\Omega(\Lambda)}$. In particular, $\pd\Sigma=\Lambda$: indeed,
			\[
			\pd\Sigma\sq\overline{\Omega(\Lambda)}\cap\pd\hypu^{n,1}=\Lambda,
			\]
		and the other inclusion follows by the fact that $\pd\Sigma$ and $\Lambda$ are graphs of functions $\pd\hyp^n\to\R$. 
		
		To prove the claim, we have to show that $\Sigma$ is smooth, spacelike and with constant mean curvature $H$, which are local properties. Hence, it suffices to prove it on $\Sigma\cap(B_R\times\R)$, for a fixed $R>0$. Since $\Sigma_k\sq I^+(W)$,
			\[
			K_R:=\overline{I^+(W)}\cap\overline{\Omega(\Lambda)}\cap(B_R\times\R)
			\]
		is a compact subset of $\overline{\Omega(\Lambda)}$ containing $\Sigma_k\cap(B_R\times\R)$, for all $k\in\N$.
			
		We recall that $I_\varepsilon(p)$ is the set $\{\dist(p,\cdot)>\varepsilon\}$ (see Definition~\ref{de:conoeps}). By definition of Cauchy hypersurface, \[\{I_\varepsilon^+(p),p\in I^-(W)\cap\Omega(\Lambda),\varepsilon>0\}\] is an open cover of $\Omega(\Lambda)$: indeed, it trivially covers $\overline{I^-(W)}\cap\Omega(\Lambda)$, while any point in the future of $W$ is connected to $W$ through a timelike past inextensible curve. We can then extract a finite subcover of $K_R$, namely $I_{\varepsilon_i}^+(p_i)$, $i=1,\dots,h$, and produce constants \[C_m(i,R)=C_m(p_i,\varepsilon_i,|H|,K_R),\] as in Proposition~\ref{pro:locest}: indeed, by Lemma~\ref{lem:invdom}, $K_R\sq I^-(p_+)$, for all $p\in\Omega(\Lambda)$. We define 
		\[C_m(R):=\max_{i=1,\dots, k}C_m(i,R),\quad m\ge-1.\]
			
		We claim that $p_i$ is eventually contained in $D(\Sigma_k)$, for all $i=1,\dots h$: hence, by Proposition~\ref{pro:locest}, the gradient function and the second fundamental form of $\Sigma_k$, together with all its derivatives, are uniformly bounded on $B_R\times\R$, for $k$ big enough. 
		
		To prove the claim, remark that $p_i$ are contained in $\Omega(\Lambda)=D(W)$, hence $I^+(p_i)$ meets $W$ in a precompact set. In particular, there exists $r>0$ such that $I^+(p_i)\cap W\sq B_r\times\R$, for all $i=1,\dots h$. In particular, if $\Sigma_k$ is a graph over $B_r$, which is the case for $k$ big enough, $I^+(p_i)\cap\Sigma_{k}$ is precompact, namely $p_i\in D(\Sigma_k)$, for all $i=1,\dots h$.
		
		The uniform bound on the gradient function ensures that $\gr (u|_{B_R})$ is spacelike, while the bounds on the derivatives of the second fundamental form imply bounds on all the derivatives of the $u_k$ on $B_R$, uniformly in $k$. We have already remarked that $|u_k|$ are uniformly bounded: hence, $\{u_{k}, k>\bar{k}\}$ is precompact in $C^{\infty}(B_R)$. Since $u_{k}$ converges to $u$ over $B_R$, $u|_{B_R}$ is smooth.
		
		\cite[Equation~(2.7)]{bartnikH} provides an explicit formula for the mean curvature of a graph $S=\gr f$, that is
			\begin{equation}\label{eq:H}
				H_S=\frac{1}{\nu_S}\left(\mathrm{div}_S(\varphi\nabla f)+\mathrm{div}_S T)\right),
			\end{equation}
		where $\mathrm{div}_S(X)=\sum_{i=1}^n\pr{\nablah_{v_i}X,v_i}$, for $v_i$ an orthonormal basis of $TS$, $X\in\Gamma(T\hypu^{n,1})$ and $\varphi:=\sqrt{-g(\pd_t,\pd_t)}$, which is known in the literature as \emph{tilt function}. The right hand side of Equation~\eqref{eq:H} is constant for $k\ge\bar{k}$, hence $H_\Sigma=H_{\Sigma_{k}}=H$ on $B_R\times\R$.
		
		Since the choice of $R$ was arbitrary, $\Sigma$ is a smooth spacelike graph with constant mean curvature $H$, which concludes the proof of the claim.
		
		\step{$H\ne0$}
		We need to exhibit a hypersurface as in the previous step.	For $H>0$, pick $\theta=\arctan(H/n)$ and choose $W=\barrp{\theta}$: Corollary~\ref{cor:diam-barriera} ensures that any compact spacelike hypersurface with constant mean curvature $H$ and whose boundary belongs to $W$ is contained in the future of $W$. The same argument applies for $H<0$, choosing $W=\barrf{\theta}$. Hence, by the previous step, there exists an entire CMC hypersurface $\Sigma_H$ bounding $\Lambda$, for any $H\ne0$.
	
		\step{$H=0$} For the maximal case, we choose one entire CMC graph just found, namely $W=\Sigma_H$, for $H\ne0$. This choice meets the conditions of the previous step by Proposition~\ref{lem:umbilical}, which concludes the proof.
	\end{proof}
	\begin{oss}
		Quite surprisingly, the maximal case is the most delicate to deal with, because, in general, $\pd_+\ch(\Lambda)=\barrp{0}$ and $\pd_-\ch(\Lambda)=\barrf{0}$ are not Cauchy hypersurfaces for $\Omega(\Lambda)$. Indeed, as soon as $\Lambda$ contains transverse lightlike segments, $\pd\ch(\Lambda)$ intersects $\pd\Omega(\Lambda)$. In that case, $\pd_\pm\ch(\Lambda)$ contains lightlike geodesic segments (Lemma~\ref{lem:light}), hence it is not an acausal hypersurface. For example, the convex core of all hypersurfaces described in Appendix~\ref{app:a} coincides with the invisible domain of their boundary (Remark~\ref{oss:convex}).
			
		To our knowledge, a direct way to overcome this problem is still to be found: indeed, \cite{bonsep,tamb} did not deal with degenerate boundaries, while \cite{ltw,sst} solve the Plateau's problem for non-degenerate boundaries and then use a compactness argument to deform them in solutions for degenerate ones.
	\end{oss}
	
	\section{A compactness result}\label{sec:comp}
The aim of this short section is to prove the following statement:

\begin{pro}\label{pro:compact}
	Let $\Sigma_k$ be a sequence of entire spacelike graphs with constant mean curvature $H_k$, contained in a precompact set of $\hypu^{n,1}\cup\pd\hypu^{n,1}$. Up to taking a subsequence, we can assume that $\Sigma_k$ converges to an entire achronal graph $\Sigma_\infty$ and $H_k$ converges to $H_\infty\in\R\cup\{\pm\infty\}$. 
	
	Then, exactly one of the following holds:
	\begin{enumerate}
		\item if $\pd\Sigma_\infty$ is not admissible, then $\Sigma_\infty$ is a totally geodesic lightlike hypersurface;
		\item if $\pd\Sigma_\infty$ is admissible and $H_\infty=\pm\infty$, then $\Sigma_\infty=\pd_\mp(\Omega(\pd\Sigma_\infty))$;
		\item if $\pd\Sigma_\infty$ is admissible and $H_\infty\in\R$, then $\Sigma_\infty$ is a CMC entire graph with constant mean curvature $H_\infty$.
	\end{enumerate}
	Moreover, in the last case, $\Sigma_k$ converges to $\Sigma_\infty$ smoothly as a graph, in any splitting.
\end{pro}
\begin{proof}
	In a splitting, $\Sigma_k$ are the graphs of $1-$Lipschitz functions $u_k$, which are uniformly bounded: indeed, any precompact set of $\hypu^{n,1}\cup\pd\hypu^{n,1}$ is contained in a suitable slice $(\hypu^{n}\cup\pd\hyp^{n})\times[a,b]$. By Ascoli-Arzelà theorem, up to extracting a subsequence, $u_k$ uniformly converges to a $1-$Lipschitz function $u_\infty\colon\hyp^n\to\R$, \emph{i.e.} $\Sigma_k$ tends to an entire achronal graph $\Sigma_\infty=\gr u_\infty$. In particular, $\Sigma_\infty\sq\overline{\Omega(\pd\Sigma_\infty)}$.
	
	\step{$\pd\Sigma_\infty$ not admissible} The boundary $\pd\Sigma_\infty$ either bounds a totally geodesic lightlike hypersurfaces or is an admissible boundary. In the former case, $\Sigma_\infty$ has to be a totally geodesic lightlike hypersurface (Lemma~\ref{lem:osc}), so we proved the first item.

	\step{$\pd\Sigma_\infty$ admissible, $H_\infty=\pm\infty$}
	Without loss of generality, $H_\infty=-\infty$. Denote $a_k:=\max_{\pd\hyp^n}|u_\infty-u_k|$, and replace $u_k$ by $u_k+a_k$, so that \[u_k|_{\pd\hyp^n}\ge u_\infty|_{\pd\hyp^n},\quad\forall k\in\N.\] In other words, $\pd\Sigma_k$ is in the future of $\pd\Sigma_\infty$: by Proposition~\ref{pro:bar}, $\Sigma_k\sq I^+(\barrf{\theta})$, for $\theta<-\arctan(H_k/n)$. It follows that $\Sigma_k$ is eventually contained in $I^+(\barrf\theta)$, for any $\theta\in(0,\pi/2)$. \[\Sigma_\infty\sq\overline{\Omega(\pd\Sigma_\infty)}\cap\bigcap_{\theta\in(0,\pi/2)}I^+\left(\barrf{\theta}\right)=\pd_+\Omega(\pd\Sigma_\infty),\]
	hence $\Sigma_\infty=\pd_+\Omega(\pd\Sigma_\infty)$ by entireness.
	
	\step{$\pd\Sigma_\infty$ admissible, $H_\infty\in\R$}
	Fix $\varepsilon>0$, denote $\h:=\sup_{k\in\N}|H_k|+\varepsilon$, which is finite as $H_k$ is bounded, and let $\Sigma_{\h}:=\gr u_{\h}$ be the unique entire spacelike graph with constant mean curvature $\h$ sharing the same boundary as $\Sigma_\infty$. As before, consider $a_k:=\max_{\pd\hyp^n}|u-u_k|$ and replace $u_k$ by $u_k+a_k$. By construction, $\pd\Sigma_k$ is in the future of $\pd\Sigma_\infty=\pd\Sigma_\h$ and $H_k<\h$. By the maximum principle (Proposition~\ref{lem:umbilical}) $\Sigma_k\sq I^+(\Sigma_\h)$, $\forall k\in\N$.
	
	In short, we will use $\Sigma_\h$ as a barrier, like in the existence proof: fix a radius $R>0$ and consider the sequence restricted to the closed cylider $B(0,R)\times\R$. \[K_R:=(B(0,R)\times\R)\cap I^+(\Sigma_{\h})\cap\Omega(\pd\Sigma_\infty)\]
	is a precompact open neighbourhood of $\Sigma_\infty\cap(B(0,R)\times\R)$ containing $\Sigma_k\cap(B(0,R)\times\R)$. As in the proof of existence, we cover $\overline{K_R}$ with a finite number of cones $I_{\varepsilon_i}^+(p_i)$, for $p_i\in I^-(\Sigma_{\h})\cap\Omega(\pd\Sigma_\infty)$, and use Proposition~\ref{pro:locest} with $L=\h$ to give a uniform bound on the gradient function and on the norm of the derivatives of the second fundamental form, in order to promote the uniform convergence to a smooth one.
	
	To conclude, one can use Equation~\eqref{eq:H} to prove that $\Sigma_\infty$ has constant mean curvature which is equal to $H_\infty$.
\end{proof}	
\subsection{A topological statement}\label{sub:top}
Denote $\mathcal{CMC}$ the space of CMC entire hypersurfaces in $\hypu^{n,1}$, equipped with the $C^\infty(\hyp^n)-$topology, and $\mathcal{B}$ the space of admissible boundaries, which is an open subspace of $\mathrm{Lip}(\sph^{n-1})$.

\begin{cor}\label{cor:compact}
	$\mathcal{CMC}$ is homeomorphic to $\mathcal{B}\times\R$.
\end{cor}
\begin{proof}
Consider the map
	\[
	\begin{tikzcd}[row sep=1ex]
		\mathcal{CMC}\arrow[r] & \mathcal{B}\times\R\\
		\Sigma\arrow[r,mapsto] & (\pd\Sigma,H_\Sigma),
	\end{tikzcd}
	\]
	
The correspondence is bijective due to Theorem~\ref{teo:A}, and is continuous and proper due to Proposition~\ref{pro:compact}, hence a homeomorphism.
\end{proof}

\begin{oss}
	The $C^{\infty}(\hyp^{n})-$topology is equivalent to the topology induced by $\mathrm{Lip}(\sph_+^n)$ on $\mathcal{CMC}$, and the latter compactifies $\mathcal{CMC}$. The boundary $\pd\mathcal{CMC}$ inside $\mathrm{Lip}(\sph_+^n)$ is described by Proposition~\ref{pro:compact}: a diverging sequence $\Sigma_k\sq\mathcal{CMC}$ converges either to a totally geodesic degenerate hypersurfaces or to the boundary of an invisible domain.
	
	The boundary of $\mathcal{B}$ in $\mathrm{Lip}(\sph^{n-1})$ consists of $1-$Lipschitz maps $f$ such that $\osc(f)=\pi$, namely boundaries of totally geodesic degenerate hypersurfaces, and we compactify $\R$ as $\R\cup\{\pm\infty\}$. It follows that the homeomorphism does not extend to the boundary of $\mathcal{B}\times\R$ in the product $\mathrm{Lip}(\sph^{n-1})\times\left(\R\cup\{\pm\infty\}\right)$. Indeed, consider the diverging sequence $(\pd\Sigma_k, H_{\Sigma_k})$: if $\pd\Sigma_k$ diverges in $\mathcal{B}$, then $\Sigma_k$ converges to a totally geodesic degenerate hypersurfaces, independently on the behaviour of $H_k$. If $\pd\Sigma_k$ converges to $\Lambda\in\mathcal{B}$, $\Sigma_k$ converges to $\pd_\pm\Omega(\Lambda)$. It follows that 
	\[
		\pd\mathcal{CMC}\cong\pd\mathcal{B}\cup\left(\mathcal{B}\times\{\pm\infty\}\right)\ne\pd(\mathcal{B}\times\R).
	\]
\end{oss}

	\section{Foliation}\label{sec:foliation}
	
We prove that entire CMC hypersurfaces analytically foliate their domain of dependence.
\begin{repthmx}{teo:C}
	Let $\Lambda$ be an admissible domain. Then $\{\Sigma_H\}_{H\in\R}$ is an analytic foliation of the invisible domain $\Omega(\Lambda)$, where $\Sigma_H$ is the unique properly embedded hypersurface with constant mean curvature equal to $H$ and boundary $\Lambda$.
\end{repthmx}
For this section, we fix an admissible boundary $\Lambda\sq\pd\hypu^{n,1}$. The section is organized as follows: first, we show that $\{\Sigma_H\}_{H\in\R}$ is a topological foliation of $\Omega(\Lambda)$. After that, we briefly present the plan to improve the regularity of the foliation. The technical computations are all contained in Subsection~\ref{sub:claim}. Finally, we prove Corollary~\ref{cor:D}.

	\subsection{Continuous foliation}\label{sub:cont}
The CMC hypersurfaces $\{\Sigma_H\}_{H\in\R}$ topologically foliate the invisible domain of $\Lambda$ if any point $p\in\Omega(\Lambda)$ is contained in a unique CMC entire hypersurface $\Sigma_H$.

The uniqueness follows from Proposition~\ref{lem:umbilical}: indeed, since they have the same boundary, $\Sigma_H$ does not intersect $\Sigma_K$, for $H\ne K$.

To prove that any point $p\in\Omega(\Lambda)$ belongs to a CMC entire hypersurface, denote
\[
H^\pm(p):=\left\{H\in\R,\,p\in I^\pm(\Sigma_H)\right\}.
\]
In the proof of Proposition~\ref{pro:compact}, we saw that $\Sigma_H$ approaches the boundary of $\pd\Omega(\Lambda)$ as $H$ diverges, namely for $H$ big enough, $p$ lies in the future of $\Sigma_H$ and in the past of $\Sigma_{-H}$. In other words $H^\pm(p)$ are not empty. By Proposition~\ref{lem:umbilical}, if $H\in H^+(p)$, then $[H,+\infty)\sq H^+(p)$. Conversely, if $H\in H^-(p)$, then $(-\infty,H]\sq H^-(p)$. It follows that \[\sup H^-(p)=\inf H^+(p)=:H(p).\] 

Finally, take a sequence $(H_k^\pm)_{k\in\N}\sq H^\pm(p)$ converging to $H(p)$: by Proposition~\ref{pro:compact}, $\Sigma_{H^\pm_k}$ converges to $\Sigma_{H(p)}$. Since $p\in I^\pm(\Sigma_{H^\pm_k})$ for all $k$, then
\[
p\in\overline{I^+(\Sigma_{H(p)})}\cap\overline{I^-(\Sigma_{H(p)})}=\Sigma_{H(p)}.
\]
\begin{oss}\label{oss:nonconvex}
	The existence of a continuous foliation provides examples of non-convex CMC entire hypersurfaces, in contrast with the flat case \cite[Corollary to Proposition~5]{trei}. The idea is the following: take an admissible boundary $\Lambda$ not asymptotic to a totally geodesic hypersurface, so that the maximal hypersurface $\Sigma_0$ is contained in the interior of $\ch(\Lambda)$. For $H$ small enough, $\Sigma_H$ intersects the interior of $\ch(\Lambda)$: if $\Sigma_H$ was convex, we could build a convex hypersurface strictly contained in the convex core, different from its boundary component, contradicting the minimality of $\ch(\Lambda)$. Furthermore, in Appendix~\ref{app:a}, we provide a class of boundaries which bound only non-convex CMC entire hypersurfaces (Remark~\ref{oss:convex}). 
\end{oss}

	\subsection{Regular foliation}\label{sub:smooth}	
In the proof of the existence part of Theorem~\ref{teo:A}, we proved that the leaves $\Sigma_H$ of the foliation are graph of smooth function on $\hyp^n$, while the map that associates $H\to\Sigma_H$ is smooth by Proposition~\ref{pro:compact}. 

Following the idea of \cite[Section~4]{bss}, we locally trivialize $\Omega(\Lambda)$ by showing that the mean curvature operator $\h$ is invertible at $\Sigma_H$, in the space of deformations of $\Sigma_H$ in $\hypu^{n,1}$. The key result needed for the proof it is the uniform bound on the norm of the derivatives of the second fundamental form (Theorem~\ref{teo:boundII}). 

By Equation~\eqref{eq:H}, if $\Sigma_H=\gr u_H$, for $u_H\in C^\infty(\hyp^{n})$, then $u_H$ satisfies the differential equation $L^H_{\hyp^n}u=0$, for
\begin{equation*}
	L_H^{\hyp^n}u=\mathrm{div}_S(\varphi\nabla u)+\mathrm{div}_S T-H\nu_S,
\end{equation*}
where $S=\gr u$, which is defined on the class of functions in $C^2(\hyp^n)$ which are $1-$Lipschitz functions with respect to the spherical metric. We recall that $\varphi=\sqrt{-g_{\hypu^{n,1}}(\pd_t,\pd_t)}$, $T=\pd_t/\phi$, and $\nu_S$ is the gradient function. The symbol of the differential operator $L^{\hyp^n}_H$ is a positive multiple of the symbol the Beltrami-Laplace operator on $S$, hence $L_H^{\hyp^n}$ is strictly elliptic. We claim that the coefficients are analytic: indeed, the divergence on $S$ can be written as
\[
\mathrm{div}_S(X)=\mathrm{div}_{\hypu^{n,1}}(X)-\pr{x,N},
\]
and an explicit computation gives
\[
\nu_S=\sqrt{1-u^2+\|\nabla u\|^2}.\]
Hence, $L^{\hyp^n}_H$ is a rational function of $u$, its first and second derivatives, $\varphi$ and $T$. By Theorem~\ref{teo:boundII} follows that the all derivatives of $u_H$ are bounded, hence $u_H\in C^{k,\alpha}(\hyp^n)$, for any $k\in\N$ and $\alpha\in(0,1)$. By \cite{hopf}, a $C^{2,\alpha}$ solution of a quasi-linear elliptic differential equation $L^{\hyp^n}_H u=0$ has the same regularity as the coefficients, namely $u_H$ is analytic. Equivalently, the leaves of the foliation are analytic.

For a fixed $H$, consider the Banach space $C^{k,\alpha}(\Sigma_H)$ (Definition~\ref{de:C2a}), for $k\in\N$ and $\alpha\in(0,1)$. Any $v\in C^{k,\alpha}(\Sigma_H)$ induces via the exponential map a deformation of $\Sigma_H$ in $\hypu^{n,1}$, which we denote $S_v$, defined by
\begin{equation*}
	\begin{tikzcd}[row sep=1ex]
		s_v\colon\Sigma_H\arrow[r,hookrightarrow] & \hypu^{n,1}\\
		p\arrow[r,mapsto] & \exp_p\left(v(p)N(p)\right).
	\end{tikzcd}
\end{equation*}
To be more explicit, in the quadric model, the map becomes
\begin{equation}\label{eq:expS}
	\left(\psi\circ s_v\right)(p)=\cos\left(v(p)\right)\psi(p)+\sin\left(v(p)\right)N(\psi(p)).
\end{equation}

In particular, for $v=0$, $S_v=\Sigma_H$, which is a spacelike entire graph. The uniform bound on $|\sff_{\Sigma_H}|$ ensures that there is an open neighbourhood $A^{k,\alpha}$ of $0\in C^{k,\alpha}(\Sigma_H)$ such that $S_v$ is a spacelike entire graph and $\pd S_v=\Lambda$, for any $v\in A^{k,\alpha}$ (Lemma~\ref{lem:A}). Thus, we define the mean curvature operator $\h\colon A^{k,\alpha}\to C^{k-2,\alpha}(\Sigma_H)$ such that $\h(v)(p)$ is the mean curvature of $S_v$ at the point $s_v(p)=\exp_p\left(v(p)N(p)\right)$.

Using \cite[Equation~(2.7)]{bartnikH}, $\h$ can be explicitly computed. Denote $\tau_H$ the submersion whose level sets are the constant normal graph over $\Sigma_H$, namely $\tau_H(x)=t$ if $x=\exp_p(tN(p))$, for some $p\in\Sigma_H$. One can compute the gradient of $\tau_H$ in the quadric model, obtaining that it is a unitary vector, hence the tilt function is just the constant function $\varphi_H\equiv1$. If $S$ is a $C^2-$spacelike hypersurface and it is a normal graph $S=\gr v$ over $\Sigma_H$, then
\begin{equation}\label{eq:Hoperator}
\h(v)=\frac{1}{\nu^H_S}\left(\mathrm{div}_S(\nabla v)+\mathrm{div}_S \nablah\tau_H\right),
\end{equation}
for $\nu^H_S=-\pr{\nablah\tau_H,N_S}$. Since $\Sigma_H$ is analytic, $\h$ is an analytic operator: one can prove it with the same argument as for $L_H^{\hyp^n}$.

We claim that $\h$ admits an analytic inverse in a neighbourhood of $v=0$ (Proposition~\ref{pro:invert}), then we define the path $v_h\sq A^{k,\alpha}$ such that $S_{v_h}=\Sigma_h$, which is well defined for $h$ in a neighbourhood of $H$ (Remark~\ref{oss:A}). The derivative of $v_h$ with respect to $h$ does not vanish at $H$ (Lemma~\ref{lem:notvanish}), that is the map
\begin{equation*}
	\begin{tikzcd}[row sep=1ex]
		S\colon\Sigma_H\times(H-\delta_{k,\alpha},H+\delta_{k,\alpha})\arrow[r]&\hypu^{n,1}\\
		(x,h)\arrow[r,mapsto]&s_{v_h}(x)
	\end{tikzcd}
\end{equation*}
is a local $C^k-$diffeomorphism onto an open neighborhood of $\Sigma_H$, \emph{i.e.} a local $C^k-$trivialization of $\Omega(\Lambda)$.
	\subsection{Proof of claims}\label{sub:claim}

We start introducing the Banach space $C^{k,\alpha}(\Sigma)$.
\begin{de}\label{de:C2a}
	Let $\Sigma$ be a complete Riemannian manifold, $k\in\N$, and $\alpha\in(0,1)$. We denote $C^{k,\alpha}(\Sigma)$ the completion of $C^{\infty}(\Sigma)$ with respect to the $(k,\alpha)-$H\"{o}lder norm, which is defined as
	\[
	\|v\|_{C^{k,\alpha}(\Sigma)}:=\max_{j\le k}\left(\sup_\Sigma |\nabla^j v|\right)+\sup_{d(x,y)<1}\frac{|\nabla^k v(x)-P_{y,x}\nabla^k v(y)|}{\mathrm{dist}(x,y)^\alpha},
	\]
	where $P_{y,x}$ is the parallel transport along the geodesic connecting $x$ and $y$.
\end{de}

\begin{lem}\label{lem:A}
	Let $\Sigma_H$ be a CMC entire hypersurface in $\hypu^{n,1}$ with constant mean curvature $H$. There exists an open neighbourhood $A^{k,\alpha}$ of $0$ in $C^{k,\alpha}(\Sigma_H)$ such that $S_v$ is a spacelike entire hypersurface and $\pd S_v=\pd\Sigma$.
\end{lem}
\begin{proof}
	We claim that if $v$ is sufficiently small in the $C^1(\Sigma)-$norm, then $S_v$ is spacelike. Incidentally, this also proves that $s_v$ is an immersion.
	
	The metric on $S_v$ can be explicitly computed using Equation~\eqref{eq:expS}: let $g_v$ be the metric on $S_v$, for $w\in T_p\Sigma$ unitary, it holds
	\begin{align*}
		(s_v^*g_v)(w,w)&=\cos(v)^2-d_pv(w)^2+\sff(w,w)\sin(2v)+\pr{B(w),B(w)}\sin(v)^2.
%		&\ge\cos(v)^2-d_pv(w)^2-|\sin(2v)||\sff|^2\\
%		&\ge\cos(v)^2-\|d_pv\|^2-|\sin(2v)|C_0(|H|,n)=1+o(v)+o(\|d_pv\|^2),
	\end{align*}
	Since $B(w)\in T\Sigma$, $\pr{B(w),B(w)}\ge0$. Moreover, $d_pv(w)^2\le\|d_pv\|^2$, for $\|\cdot\|$ the operator norm in $\mathrm{Hom}(T_p\Sigma,\R)$. Finally, $\sff(w,w)\le C_0(|H|,n)$ (Theorem~\ref{teo:boundII}). It follows that
	\[
		(s_v^*g_v)(w,w)\ge\cos(v)^2-\|d_pv\|^2-|\sin(2v)|C_0(|H|,n)=1+o\left(|v|\right)+o\left(\|d_pv\|^2\right),
	\]
	which proves the claim.
	
	To conclude, we claim that, for $h$ close enough to $H$, $\Sigma_h$ can be written as a normal graph over $\Sigma_H$. In particular, 
	there exist $H_1<H<H_2$ such that $\Sigma_{H_i}$ is the normal graph of $v_{H_i}$ over $\Sigma_H$. By Proposition~\ref{lem:umbilical}, $\Sigma_{H_1}\sq I^+(\Sigma)$ and $\Sigma_{H_2}\sq I^-(\Sigma)$, that is $v_{H_2}<0<v_{H_1}$: it follows that for any $v\in C^{k,\alpha}$ such that $v_{H_2}<v<v_{H_1}$, then $\pd S_v=\pd\Sigma$. Since $S_v$ is spacelike and properly immersed, is properly embedded by \cite[Lemma~4.5.5]{bonsep}.
	
	We prove the claim by contradiction: assume there exist $\varepsilon>0$ and a sequence $h_k\to H$, such that $\Sigma_{h_k}$ is not contained in the $\varepsilon-$normal neighborhood of $\Sigma_H$, that is there exists a sequence of points $p_k\in\Sigma_H$ such that $\dist(p_k,\Sigma_{h_k})>\varepsilon$, for any $k\in\N$. For any $k\in\N$, choose an isometry $f_k$ of $\hypu^{n,1}$ sending $p_k$ to $(x_0,0)$ and $N_{\Sigma_H}(p_k)$ to the normal vector to $\hyp^{n}\times\{0\}$ at $(x_0,0)$. Remark that $f_k(\Sigma_H)$ and $f_k(\Sigma_{h_k})$ share the same boundary for any $k\in\N$: by Proposition~\ref{pro:compact}, up to extacting a subsequence, they converge to the same acausal graph, which can be either an $H-$hypersurface or a totally geodesic degenerate hyperplane. The choice of the normal vector of $f_k(\Sigma_H)$ at $(x_0,0)$ prevents the latter to happen, hence 
	\[\dist(p_k,\Sigma_{h_k})=\dist\left((x_0,0),f_k(\Sigma_{h_k})\right)\to0,\]
	which contradicts the assuption, proves the claim and concludes the proof.	
\end{proof}
\begin{oss}\label{oss:A}
	In particular, the leaves of the foliation are contained in $A^{k,\alpha}$, if their mean curvature is sufficiently close to $H$.
\end{oss}

\begin{pro}\label{pro:invert}
	Let $\Sigma$ be a CMC entire spacelike hypersurface. The operator $\h$ on $C^{k,\alpha}(\Sigma)$ admits an analytic inverse in a neighbourhood of $v=0$.
\end{pro}
\begin{proof}
	To prove that $\h$ is invertible at $v=0$, we first linearize it: denote $J$ the linearization of the mean curvature operator $\h$ at $0$. Since $\h$ is analytic, by analytic inverse function theorem its local inverse is analytic, too. By \cite[Lemma~7.3]{sst},
	\begin{equation}\label{eq:J}
		J=\Delta-n-|\sff|^2,
	\end{equation}
	for $\Delta$ the Laplace-Beltrami operator on $\Sigma$.	Our goal is to build a bounded inverse $J^{-1}$ at $0$ in $C^{k,\alpha}(\Sigma)$. 
	
	\step{$k=2$} 
	The existence of an inverse is equivalent to prove that the differential problem $Ju=f$ has always solution, and that any solution satisfies a Schauder-type inequality
	\begin{equation}\label{eq:schau}
		\|u\|_{C^{2,\alpha}(\Sigma)}\le C\|f\|_{C^{0,\alpha}(\Sigma)},
	\end{equation}
	for some constant $C>0$ not depending on $f$. Indeed, Equation~\eqref{eq:schau} then implies that $J$ is injective, hence invertible, and $J^{-1}$ is bounded.
	First, we build a solution for $Ju=f$, for a fixed $f\in C^{0,\alpha}(\Sigma)$. Since $\Sigma$ is complete (Theorem~\ref{teo:B}), we can pull-back the problem on $\R^n$ via the exponential map, namely in normal coordinates around a point. By Equation~\eqref{eq:J}, $J$ is strictly elliptic: by \cite[Theorem~6.14]{pde} there exists a unique $C^{2,\alpha}(\overline{K_i})$ solution $u_i$ to the Dirichlet problem
	\begin{equation*}
		\begin{cases}
			Ju=f|_{K_i}\\
			u|_{\pd K_i}=0
		\end{cases}
	\end{equation*}
	for $\{K_i\}_{i\in\N}$ an exhaustion of compact sets of $\Sigma$. We claim that there exists a local version of Equation~\eqref{eq:schau}, namely
	\begin{equation}\label{eq:schau2}
		\|u_i\|_{C^{2,\alpha}(K_i)}\le C\|f\|_{C^{0,\alpha}(\Sigma)},
	\end{equation}
	where $C$ does not depend on $i$ nor $f$. Hence, we can apply Ascoli-Arzelà Theorem to extract a subsequence converging in $C^{2,\alpha}(\Sigma)$ to a global solution $u$. In particular, the limit $u$ satisfies Equation~\eqref{eq:schau}, concluding the proof.
	
	To prove the claim, fix $x\in\Sigma$. In normal coordinates around $x$, $J$ is a uniformly strictly elliptic operator on bounded sets of $\R^n$, in particular on the ball $B(0,2)$. By \cite[Theorem~6.2]{pde}, we obtain
	\begin{equation}\label{eq:schau3}
		\|u\|_{C^{2,\alpha}\left(B(0,1)\right)}\le\tilde{C}\left(\|u\|_{C^0\left(B(0,2)\right)}+\|f\|_{C^{0,\alpha}\left(B(0,2)\right)}\right),
	\end{equation}
	where $\tilde{C}$ depends on the uniform bounds of ellipticity of $J$ in $B(0,2)$, hence ultimately $\tilde{C}$ depends on $x$. Actually, by the uniform bound on the norm of the derivatives of $\sff$, the pull-back of $J$ is strictly elliptic uniformly with respect to $x$, hence one can choose $\tilde{C}$ holding for all $x\in\Sigma$. Moreover, for $v=c$ a constant function, \[|J(v)|=|c|(n+|\sff|^2)\le|c|\left(n+C_0(|H|,n)\right),\]
	for $C_0(|H|,n)$ as in Theorem~\ref{teo:boundII}. It follows that the constant functions
	\begin{align*}
		u_+:=\frac{\|f\|_{C^{0,\alpha}\left(B(0,2)\right)}}{n+C_0(|H|,n)} &&		u_-:=-\frac{\|f\|_{C^{0,\alpha}\left(B(0,2)\right)}}{n+C_0(|H|,n)}
	\end{align*}
are respectively a supersolution and a subsolution for $J$. By the strong maximum principle (\cite[Theorem~3.5]{pde}), $u_-<u<u_+$ on $B(0,2)$, namely
\[
	\|u\|_{C^0\left(B(0,2)\right)}<\frac{\|f\|_{C^{0,\alpha}\left(B(0,2)\right)}}{n+C_0(|H|,n)}\le\frac{\|f\|_{C^{0,\alpha}\left(\Sigma\right)}}{n+C_0(|H|,n)}.
\]

Substituting in Equation~\eqref{eq:schau3}, one obtains
\[
	\|u_i\|_{C^{2,\alpha}\left(B(0,1)\right)}\le\tilde{C}\left(\frac{1}{n+C_0(|H|,n)}+1\right)\|f\|_{C^{0,\alpha}(\Sigma)}=:C\|f\|_{C^{0,\alpha}(\Sigma)},
\]
which proves Equation~\eqref{eq:schau2}, hence the claim, concluding the proof for $k=2$.  

\step{$k>2$}
It suffices to repeat the argument above for the higher derivatives, remarking that $J$ commutes with the derivatives: let $\beta=(i_1,\dots,i_{|\beta|})$ be a multi index of lenght $|\beta|\le k-2$, that is $D^\beta=\pd_{i_1}\dots\pd_{i_{|\beta|}}$. Since $D^\beta u$ is a solution of $Jv=D^\beta f$, the same argument as above implies 
	\[\|D^\beta u\|_{C^{2,\alpha}(\Sigma)}\le C\|D^\beta f\|_{C^{0,\alpha}(\Sigma)},\]
hence 
	\[\|u\|_{C^{k,\alpha}(\Sigma)}\le\sum_{\beta,\,|\beta|\le k-2}\|D^\beta u\|_{C^{2,\alpha}(\Sigma)}\le\sum_{\beta,\,|\beta|\le k-2}C\|D^\beta f\|_{C^{0,\alpha}(\Sigma)}=C\|f\|_{C^{k-2,\alpha}(\Sigma)},\]
which proves that $J$ is invertible at $0$ in $C^{k,\alpha}$.
\end{proof}

The following lemma allows us to apply the analytic inverse function theorem to the smooth path $v_\bullet\colon(H-\delta,H+\delta)\to C^{k,\alpha}(\Sigma_H)$ such that $S_{v_h}=\Sigma_h$.
\begin{lem}\label{lem:notvanish}
	The derivative of $v_h$ with respect to $h$ does not vanish at $H$.
\end{lem}
\begin{proof}
	By construction, $\h(v_h)=h$. Differentiating both sides, one obtains
	\[
	J\left(\frac{dv_h}{dh}\right)=1.
	\]
	
	We recall that, following the proof of Proposition~\ref{pro:invert}, we can write $dv_h/dh$ as the limit of functions $v_i$ which are solutions of the differential problem
		\begin{equation*}
		\begin{cases}
			Jv=1|_{K_i}\\
			v|_{\pd K_i}=0
		\end{cases}
	\end{equation*} 
for $K_i$ an exhaustion of compact set. 

The constant function $u=0$ is a supersolution for the equation, since $J(u)=0$. As already remarked, $J$ is a strictly elliptic operator on $\Sigma$ (Equation~\eqref{eq:J}), hence uniformly elliptic over compact set. By the weak maximum principle (\cite[Theorem~3.1]{pde}), the maximum of $v_i$ is reached at the boundary, namely $v_i\le0$ for any $i\in\N$, hence $dv_h/dh\le0$. We can then apply the strong maximum principle (\cite[Theorem~3.5]{pde}) to obtain $dv_h/dh<u=0$, which concludes the proof.
\end{proof}

	\subsection{Analytic foliation}\label{sub:anal}
In the previous section, we have proved that the $v_h$ is an analytic map, and that the path $h\to v_h$ is analytic in $C^{k,\alpha}$. Since the evaluation at a point $p\in\Sigma_H$ is an analytic operator, the map $h\mapsto v_h(p)$ is an analytic map. It follows that $v_\bullet(\cdot)$ is analytic both in the argument $p\in\Sigma$ and $h\in\R$. To prove that the map is jointly analytic, it suffices to pullback the problem using the exponential map, to see $v_\bullet(\cdot)$ as a map $\R^n\times\R\to\R$, and prove that the radius of convergence in both variables are uniformly bounded from below (\cite[Theorem~(I)]{anal}). 

Since $h\to v_h$ is an analytic path, the radius of convergence $r$ of $v_\bullet(p)$ at $h=H$ does not depend on $p$, hence $r/2$ is a uniform lower bound for the radius of convergence of $v_\bullet(p)$ at $h$, for $h\in(H-r/2,H+r/2)$, $p\in\Sigma_H$. Pick $p\in\Sigma_H$, we claim that there exists $\rho,\varepsilon>0$ such that $v_h$ has radius of convergence at least $\rho$ at $p$, for any $h\in(H-\varepsilon,H+\varepsilon)$. As a consequence, for any such $h$, the radius of convergence of $v_h(\cdot)$ at $q$ is bounded from below by $\rho/2$, for any $q\in B_\Sigma(p,\rho/2)$. It follows that
\[\delta:=\min\{r/2,\rho/2,\varepsilon\}\]
is a uniform lower bound for the radii of convergence of both $v_\bullet(q)$ at $h$ for $h\in(H-\varepsilon,H+\varepsilon)$, and $v_h(\cdot)$ at $q$, for $q\in B_\Sigma(p,\rho/2)$. We conclude that 
\[v_\bullet(\cdot)\colon B_\Sigma(p,\rho/2)\times(H-\varepsilon,H+\varepsilon)\to\Omega(\Lambda)\] 
is a local analytic trivialization of $\Omega(\Lambda)$. Since $p$ and $H$ are arbitrary, this concludes the proof.

We remark that by Equation~\eqref{eq:Hoperator}, any $v_h$ solves the analytic non-linear elliptic differential equation $L_h^{\Sigma_H}v=0$, for 
\begin{equation*}
	L_h^{\Sigma_H}v:=\Delta^{S_v} v+\mathrm{div}_{S_v}(\nablah\tau_H)-h\nu^{H}_{S_v},
\end{equation*}
Without loss of generality, we can assume $p=0$, so that the claim can be then proved by following the proof of analyticity of \cite{hopf}, which consists in building a complex extention of the solution $v$ in a neighbourhood of $0$ and prove that it is analytic on the set 
\[(R)_\gamma:=\{z=x+iy\in\C^n,\|x\|<R,\|y\|<\gamma(R-\|x\|)\},\]
(see \cite[Page~221]{hopf}), for $\gamma$ a constant depending continuosly on the symbol of the operator (see \cite[Equation~(6.8)]{hopf}), which analytically depends on the function $u$ and its derivatives up to the fourth. On the other hand, $R$ has to be bounded from above by many quantities, which continuously depend on the function $u$ and its derivatives up to the sixth (see \cite[Equations~(6.14),(7.9),(8.6),(8.7)]{hopf}). 

\emph{A priori}, $R=R(h)$ and $\gamma=\gamma(h)$. However, since the symbol of $L_h^{\Sigma_H}$ does not depend on $h$, as a function of $u$ and its derivatives up to the second, and $v_h$ is a $C^6-$foliation, for $h\in(H-\delta_{6,\alpha},H+\delta_{6,\alpha})$, we can find $\varepsilon<\delta_{6,\alpha}$ such that all forementioned quantities are uniformly bounded: it follows that \[\bar{R}:=\inf_{h\in(H-\varepsilon,H+\varepsilon)}R(h)>0,\qquad\bar{\gamma}:=\inf_{h\in(H-\varepsilon,H+\varepsilon)}\gamma(h)>0.\] 
Hence, the complex extention of $v_h$ is analytic on $(\bar{R})_{\bar{\gamma}}$, for all $h\in(H-\varepsilon,H+\varepsilon)$. Setting $\rho:=\min\{\bar{R},\bar{\gamma}\bar{R}\}/2$, the ball $B_{\C^n}(0,\rho)$ is contained in $(\bar{R})_{\bar{\gamma}}$, hence the radius of convergence of $v_h$ at $p$ is at least $\rho$, proving the claim and concluding the proof.

	\subsection{Maximal globally hyperbolic Cauchy complete $\AdS-$manifolds}
	
This section is meant to extend \cite[Theorem~1.5]{abbz} from maximal globally hyperbolic Cauchy \emph{compact} $\AdS-$manifolds to maximal globally hyperbolic Cauchy \emph{complete} $\AdS-$manifolds, namely

\begin{repcorx}{cor:D}
	Let $(M,g)$ be a maximal globally hyperbolic Cauchy complete Anti-de Sitter manifold. Then $(M,g)$ admits a (unique) globally defined CMC time function $\tau_{cmc}\colon M\to\R$.
\end{repcorx}

\begin{de}
A globally hyperbolic $\AdS-$manifold is called
\begin{itemize}
	\item \textit{Cauchy complete} if it admits a Cauchy hypersurface whose induced Riemannian metric is complete.
	\item \textit{maximal} if every isometric embedding $M\hookrightarrow N$ in another globally hyperbolic $\AdS-$manifold is an isometry.
\end{itemize}
\end{de}
\begin{de}
	A time function $\tau$ on an time-oriented Lorentzian manifold is called a \textit{CMC} time function if each level set $\tau^{-1}(H)$ is a hypersurface with constant mean curvature $H$.
\end{de}

\begin{oss}\label{oss:cmctime}
	The function $\tau_{cmc}\colon\Omega(\Lambda)\to\R$ which associates to each point $p\in\Omega(\Lambda)$ the unique $H$ such that $p\in\Sigma_H$ is a CMC time function: by definition, $\tau_{cmc}^{-1}(H)=\Sigma_H$, and it is strictly decreasing along future-directed time paths due to the strong maximum principle (Proposition~\ref{lem:umbilical}).
\end{oss}

The proof reduces to rephrase Theorem~\ref{teo:C} in this setting, using the classification provided by \cite[Proposition~6.3.1, Corollary~6.3.13]{benbon}:
\begin{pro}\label{pro:globhyp}
	Let $(M,g)$ be a maximal globally hyperbolic Cauchy complete $\AdS-$manifold, then its universal cover $\widetilde{M}$ isometrically embeds in $\hypu^{n,1}$ and its image is the invisible domain of an admissible boundary $\Lambda$. Moreover, $\Lambda$ is unique, up to isometry of $\hypu^{n,1}$.
	
	Conversely, $\Omega(\Lambda)$ is a maximal globally hyperbolic Cauchy complete $\AdS-$manifold, for any admissible boundary $\Lambda\sq\pd_\infty\hypu^{n,1}$.
\end{pro}

\begin{oss}\label{oss:globhyp}
	It follows that any maximal globally hyperbolic Cauchy complete $\AdS-$manifold can be written as $\Omega(\Lambda)/\Gamma$, for $\Gamma$ a subgroup of $\Isom(\hypu^{n,1})$. Since $M$ is globally hyperbolic, $\Gamma$ consists of time-orientation preserving isometries: otherwise, $M$ would not be time-orientable, and in particular not globally hyperbolic.
\end{oss}

\begin{proof}[Proof of Corollary~\ref{cor:D}]
	Remark~\ref{oss:cmctime} proves the statement for $(M,g)$ simply connected. If $\pi_1(M)$ is not trivial, by Proposition~\ref{pro:globhyp}, $(M,g)=\Omega(\Lambda)/\Gamma$, for some $\Lambda$ admissible boundary and $\Gamma\sq\Isom(\hypu^{n,1})$. We claim that the CMC time function on $\Omega(\Lambda)$ is invariant over the orbits of $\Gamma$.
	
	First, remark that $\Lambda$ is $\Gamma-$invariant, hence $\Sigma_H$ is also $\Gamma-$invariant: indeed, for $g\in\Gamma$, $g(\Sigma_H)$ is a CMC entire hypersurface whose boundary is $g(\Lambda)=\Lambda$. Remark~\ref{oss:globhyp} ensures $g$ is time-orientation preserving, hence the mean curvature of $g(\Sigma_H)$ is $H$. By uniqueness, $g(\Sigma_H)$ coincides with $\Sigma_H$. Since the CMC time function on $\Omega(\Lambda)$ associates to a point the mean curvature of the unique CMC entire hypersurface it belongs to, $\tau_{cmc}$ is $\Gamma-$invariant, which concludes the proof.
\end{proof}

\appendix
	
	\section{Explicit bounds}\label{app:a}

Combining \cite[Theorem~1]{kkn} and Theorem~\ref{teo:B}, we can sharpen the bound $C_0(L,n)$ on the norm second fundamental form (Theorem~\ref{teo:boundII}).
	
	\begin{teo}\label{teo:expbound}
		Let $L\ge0$. For any properly embedded hypersurface $\Sigma$ with constant mean curvature $H\in[-L,L]$ in $\hypd^{n,1}$, the following holds:
		\begin{equation}\label{eq:bound}
			|\sff_\Sigma|^2\le n\left(1+\frac{L^2+L(n-2)\sqrt{L^2+4(n-1)}}{2(n-1)}\right).
		\end{equation}
		
		Moreover, if the maximum is reached at one point, $\sff_\Sigma$ is parallel.
	\end{teo}
The proof consists in replacing the word \textit{complete} with \textit{properly embedded} in the statement \cite[Theorem~1]{kkn}, which is possible because Theorem~\ref{teo:B} makes the two properties equivalent for CMC hypersurfaces.

In the following, we classify properly embedded hypersurfaces with parallel second fundamental form (Proposition~\ref{pro:parallel}) in order to present \cite[Theorem~2]{kkn} from a more geometric point of view (Proposition~\ref{pro:HpxHq}).

\subsection{Cylindrical hypersurfaces}

Consider a totally geodesic spacelike submanifold $M$ in $\hypd^{n,1}$, and let $M^{\perp}_+$ be its dual in the future, namely
\[M^{\perp}_+:=\bigcap_{x\in M} P_+(x).\]

Each pair $(x,y)\in M\times M^{\perp}_+$ is connected by the timelike geodesic \[\gamma_{xy}(\theta):=\cos(\theta)x+\sin(\theta)y.\]
\begin{de}
	Let $M$ be a totally geodesic spacelike submanifold of dimension $k\in\{0,\dots,n\}$, and $\theta\in(0,\pi/2)$. We define a cylindrical hypersurface \[\hyp(k,\theta):=\{\gamma_{xy}(\theta),x\in M,y\in M^{\perp}_+\}.\]
\end{de}
Let $M'$ be another totally geodesic spacelike submanifold of dimension $k$. One can easily check that any time preserving isometry sending $M$ to $M'$ sends $\hyp(k,\theta)$ to 
\[\{\gamma_{xy}(\theta),x\in M',y\in(M')^{\perp}_+\},\]
namely $\hyp(k,\theta)$ is well defined, up to isometry. Moreover a time reversing isometry fixing $M$ sends $\hyp(k,\theta)$ to $\hyp\left(n-k,\frac{\pi}{2}-\theta\right)$: indeed, it sends $M_+$ to $M_-$, and $(M_-)_+=M$.
\begin{oss}\label{oss:H=P}
	For $k=0,n$, we recover the equidistant hypersurfaces described in Subsection~\ref{sub:equidistant}: $P^-_\theta=\hyp(n,\theta)$ and $P^+_\theta=\hyp(0,\theta)$.
\end{oss}

The following lemmas describes the geometry of the cylindrical hypersurfaces. A direct computation gives:

\begin{lem}\label{lem:HpxHq}
	$\hyp(k,\theta)$ is a properly embedded spacelike hypersurface isometric to 
	\[(\cos\theta)\hyp^k\times(\sin\theta)\hyp^{n-k},\]
	whose boundary consists of lightlike segments connecting the boundaries of $\pd\hyp^k$ and $\pd\hyp^{n-k}$.
\end{lem}

\begin{lem}\label{lem:horoII}
	$\hyp(k,\theta)$ has parallel second fundamental form. In particular, it has constant mean curvature $H=k\tan(\theta)-\frac{n-k}{\tan(\theta)}$.
\end{lem}
\begin{proof}
	To prove that $\sff$ is parallel, it suffices to remark that 
	\[\Isom(\hyp^{k})\times\Isom(\hyp^{n-k})\cong\mathrm{O}(k,1)\times\mathrm{O}(n-k,1)\sq\mathrm{O}(n,2)\cong\Isom(\hypd^{n,1}),\]
	acts transitively on $\hyp(k,\theta)$. By a direct computation, analogous as the one contained in the proof of Lemma~\ref{lem:equi}, it holds
	\begin{equation}\label{eq:horoII}
		B=\left(\tan\theta\mathrm{Id}_k,-\frac{1}{\tan\theta}\mathrm{Id}_{n-k}\right),
	\end{equation}
	which concludes the proof.
\end{proof}

\begin{oss}\label{oss:convex}
	For $k\in\{1,\dots n-1\}$, one can prove $\ch(\Lambda_k)=\overline{\Omega(\Lambda_k)}$, for $\Lambda_k:=\pd\hyp(k,\cdot)$. It follows that cylindrical hypersurfaces are not convex, for $k\ne0,n$. By contradiction, assume $\hyp(k,\theta)$ is convex, for $\theta\in(0,\pi/2)$: hence, either $I^-(\hyp(k,\theta))\cap I^+(\pd_-\ch(\Lambda_k))$ or $I^+(\hyp(k,\theta))\cap I^-\left(\pd_+\ch(\Lambda_k)\right)$ is a convex set containing $\Lambda_k$ and strictly contained in $\ch(\Lambda_k)$, contradicting the minimality of the convex hull.
\end{oss}

\subsection{Achieving the bound} We exhibit the only CMC entire hypersurfaces achieving the bound of Theorem~\ref{teo:expbound}.

We recall some results of pseudo-Riemannian geometry, which we state for $\hypd^{n,1}$, but hold in complete generality. Since the second fundamental form $\sff$ and the shape operator $B$ are dual with respect to the metric, $\sff$ is parallel if and only if $B$ is parallel. Moreover, $B$ is a symmetric $(1,1)-$tensor, hence diagonalizable with respect to an orthonormal basis.
\begin{lem}\label{lem:parallel}
	Let $\Sigma$ be a spacelike hypersurface of $\hypd^{n,1}$ with parallel shape operator $B$,
	\begin{enumerate}
		\item the eigenvalues of $B$ are constant along $\Sigma$;\label{it:p1}
		\item the eigenspaces of $B$ are parallel. More precisely, let $V_\lambda$ be the distribution such that $V_\lambda(x)\sq T_x\Sigma$ is the eigenspace of $\lambda$ at $x$: then $\nabla_w V_\lambda\sq V_\lambda$, for any $w\in T_x\Sigma$;\label{it:p2}
		\item the eigenspaces of $B$ are integrable.\label{it:p3}
	\end{enumerate} 
\end{lem}
\begin{proof}
	We recall that $B$ is parallel if and only if 
	\begin{equation}\label{eq:Bparallel}
		\nabla_Y B(X)=B(\nabla_Y X),\qquad \forall X,Y\in\Gamma(T\Sigma).
	\end{equation}
	Let $X\in\Gamma(T\Sigma)$ be a unitary eigenvector field, namely $B(X)=\lambda X$, for $\lambda\in C^\infty(\Sigma)$. To prove \eqref{it:p1}, it suffices to show that $Y(\lambda)=0$, for all $Y\in\Gamma(T\Sigma)$. By Equation~\eqref{eq:Bparallel}, one obtains 
	\begin{align*}
		Y(\lambda)=\nabla_Y\pr{B(X),X}=\pr{B(\nabla_Y X),X}+\pr{B(X),\nabla_Y X}.
	\end{align*}
	Since $B$ is symmetric with respect to the metric,
	\begin{align*}
		Y(\lambda)=2\pr{\nabla_Y X,B(X)}=2\lambda\pr{\nabla_Y X,X},
	\end{align*}
	which vanishes since $X$ is unitary.
	
	To prove \eqref{it:p2}, substitute \eqref{it:p1} in Equation~\eqref{eq:Bparallel} to obtain
	\[
	B(\nabla_Y X)=\nabla_Y B(X)=Y(\lambda)X+\lambda\nabla_Y X=\lambda\nabla_Y X,
	\] 
	that is $\nabla_Y X$ is an eigenvector field with respect to $\lambda$, \emph{i.e.} $\nabla_Y X\in V_\lambda$. 
	
	Finally, \eqref{it:p3} follows directely by Frobenius Theorem: indeed, for any $X,Y\in V_\lambda$
	\[[X,Y]=\nabla_X Y-\nabla_Y X,\]
	which belongs to $V_\lambda$ by \eqref{it:p2}, concluding the proof.
\end{proof}
\begin{pro}\label{pro:parallel}
 	Cylindrical hypersurfaces are the only properly embedded hypersurfaces in $\hypd^{n,1}$ with parallel second fundamental form.
\end{pro}
\begin{proof}
	Let $\Sigma$ be a properly embedded hypersurface of $\hypd^{n,1}$ with parallel second fundamental form. In particular, it is an entire graph, hence diffeomorphic to $\R^n$. By the fundamental theorem of immersed hypersurfaces, it suffices to prove that $\Sigma$ has the same induced metric as $\hyp(k,\theta)$, for a suitable choice of $k=0,\dots,n$ and $\theta\in(0,\pi/2)$, and that they have the same shape operator.
	
	The eigenspaces $V_{\lambda_i}$ of $B$ are integrable (Lemma~\ref{lem:parallel}) and parallel, moreover, $\Sigma$ is a CMC hypersurface, hence complete (Theorem~\ref{teo:B}). By De Rham Decomposition Theorem, $\Sigma$ is isometric to the product of the integral submanifolds $M_i$ of $V_{\lambda_i}$ (see \cite[Chapter~III, Section~6]{kobanom1}). Moreover, $M_i$ is a complete totally geodesic submanifold of $\Sigma$.
	
	Since $\Sigma$ is a product, the sectional curvature vanishes along any tangent 2-plane which is not contained in an eigenspace $V_{\lambda_i}$. On the other hand, by Gauss equation, the sectional curvature along the tangent 2-plane $\Span\{v_i,v_j\}$, for $v_i\in V_{\lambda_i}$ and $v_j\in V_{\lambda_j}$, is
	\begin{equation*}%\label{eq:sectcurv}
		K_\Sigma\left(\Span\{v_i,v_j\}\right)=-1-\lambda_i\lambda_j.
	\end{equation*}
	It follows that $M_i$ is a simply connected complete manifold with constant sectional curvature $-1-\lambda_i^2<0$ and dimension $k_i=\dim V_{\lambda_i}$, hence it is isometric to $\cos(\theta_i)\hyp^{k_i}$, for $\theta_i=\arctan(\sqrt{\lambda_i})$ (compare with Lemma~\ref{lem:equi}). Moreover, since the sectional curvature vanishes for $i\ne j$, $B$ has at most two eigenvalues, and in that case $\lambda_2=-1/\lambda_1$. 
	
	It follows that $\Sigma$ the metric and the shape operator of $\Sigma$ coincide with the ones of $\hyp(k_1,\theta_1)$ (Equation~\eqref{eq:horoII}), which concludes the proof.
\end{proof}

\begin{pro}\label{pro:HpxHq}
	Let $\Sigma$ be a properly embedded hypersurface with constant mean curvature $H$. Assume there exists $x\in\Sigma$ such that \[|\sff_\Sigma|^2(x)=n\left(1+\frac{H^2+|H|(n-2)\sqrt{H^2+4(n-1)}}{2(n-1)}\right).\]
	\begin{itemize}
		\item If $H=0$, then $\Sigma\cong\hyp\left(k,\arctan(\sqrt{(n-k)/k})\right)$, for some $k=1,\dots, n-1$;
		\item otherwise, $\Sigma\cong\hyp(1,\theta_H)$,
	\end{itemize}
where $\tan(\theta_H)$ is the positive solution of $t^2-Ht-(n-1)=0$.
\end{pro}
\begin{proof}
	 Denote $S$ the bound. By Theorem~\ref{teo:expbound}, if the bound is achieved at one point, then the second fundamental form is parallel, hence $\Sigma$ is a cylindrical hypersurface by Proposition~\ref{pro:parallel}. One can then compute explicitely the norm of the second fundamental form of $\hyp(k,\theta)$, using Equation~\eqref{eq:horoII}.

	 Denote $S(k,\theta):=|\sff_{\hyp(k,\theta)}|^2$ and $S$ the bound in the statement. For $k=0,n$, namely $\Sigma=P^\pm_\theta$, $S(k,\theta)=H^2/n<S$. For $k\ne0,n$, one obtains
	 \[
	 	S(k,\theta)=n+\frac{nH^2+|H|(n-2k)\sqrt{H^2+4k(n-k)}}{2k(n-k)}.
	 \]
	 In the maximal case, $S(k,\theta)=S$ for any $k\in\{1,\dots,n-1\}$. For $H\ne0$, $S(k,\theta)=S$ if and only if at $k=1,n-1$, concluding the proof.
\end{proof}

\bibliographystyle{alpha}
\bibliography{CMC_biblio.bib}

\end{document}